\documentclass[a4paper,11pt]{article}

 \usepackage{hyperref} 
 
\usepackage[activate={true,nocompatibility},final,tracking=true,kerning=true,spacing=true,factor=1100,stretch=10,shrink=10]{microtype}

\SetProtrusion{encoding={*},family={bch},series={*},size={6,7}}
              {1={ ,750},2={ ,500},3={ ,500},4={ ,500},5={ ,500},
               6={ ,500},7={ ,600},8={ ,500},9={ ,500},0={ ,500}}
               
\SetExtraKerning[unit=space]
    {encoding={*}, family={bch}, series={*}, size={footnotesize,small,normalsize}}
    {\textendash={400,400},  
     "28={ ,150}, 
     "29={150, },  
     \textquotedblleft={ ,150},  
     \textquotedblright={150, }}  
     
\SetExtraKerning[unit=space]
   {encoding={*}, family={qhv}, series={b}, size={large,Large}}
   {1={-200,-200}, 
    \textendash={400,400}}

\usepackage{setspace}  

\usepackage[usenames,dvipsnames]{xcolor}  

\usepackage[fleqn]{amsmath} 
\usepackage{amsthm}

\usepackage{amssymb}
\usepackage[affil-it]{authblk}

\usepackage{bbm}  
\usepackage{booktabs} 
\usepackage{bm} 

\usepackage[justification=justified, font=small, format=hang, margin={1cm,1cm}]{caption}  
\usepackage{color}
\usepackage{calligra}

\usepackage{dsfont}

\usepackage{enumitem} 
\usepackage[T1]{fontenc} 

\usepackage[left=3.3cm, right=3.3cm, top=3.3cm, bottom=3.3cm, showframe=false]{geometry} 

\usepackage[utf8]{inputenc} 

\usepackage{lmodern}
\usepackage{lscape}

\usepackage{makeidx}
\usepackage{mathtools}
\usepackage{mathrsfs}  

\usepackage[round]{natbib}

\usepackage{subfig}  

\usepackage{thmtools} 
\usepackage[noindentafter]{titlesec}
 
\titleformat{\chapter}[block]{\filcenter\bfseries}{\thechapter.}{0.5em}{}
\titlespacing*{\chapter}{0pt}{2\baselineskip}{1\baselineskip}
\titleformat{\section}[block]{\filcenter\scshape}{\thesection.}{0.5em}{}
\titlespacing*{\section}{0pt}{2\baselineskip}{1\baselineskip}
\titleformat{\subsection}[block]{\filcenter\scshape}{\thesubsection.}{0.5em}{}
\titlespacing*{\subsection}{0pt}{2\baselineskip}{1\baselineskip}
\titleformat{\subsubsection}[block]{\filcenter\scshape}{\thesubsubsection.}{0.5em}{}
\titlespacing*{\subsubsection}{0pt}{2\baselineskip}{1\baselineskip}

\usepackage{xcolor}

\newcommand{\arxiv}{{arXiv}}

\newcommand{\computationalstatisticsanddataanalysis}{{Computational Statistics and Data Analysis}}

\newcommand{\journalofeconometrics}{{Journal of Econometrics}} 

\newcommand{\jspi}{{Journal of Statistical Planning and Inference}}  

\newcommand{\journalmultvaranal}{{Journal of Multivariate Analysis}} 

\newcommand{\journalroyalsociety}{{Journal of the Royal Statistical Society: Series B}} 

\newcommand{\stochprocappl}{{Stochastic Processes and their Applications}}

\newcommand{\statprob}{{Statistics and Probability Letters}}

\newcommand{\annmathstat}{{The Annals of Mathematical Statistics}}
 
\newcommand{\annstat}{{The Annals of Statistics}}

\newcommand{\cG}{{\cal G}}

\newcommand{\cL}{{\cal L}}

\newcommand{\cO}{{\mathcal O}}

\newcommand{\cS}{{\cal S}}

\newcommand{\ccB}{{\mathscr B}}

\newcommand{\ccC}{{\mathscr C}}
\newcommand{\ccD}{{\mathscr D}}

\newcommand{\ccG}{{\mathscr G}}

\newcommand{\ccK}{{\mathscr K}}
\newcommand{\ccL}{{\mathscr L}}

\newcommand{\ccS}{{\mathscr S}}
\newcommand{\ccT}{{\mathscr T}}

\newcommand{\mN}{\mathbb{N}}
\newcommand{\mR}{\mathbb{R}}
\newcommand{\mZ}{\mathbb{Z}}

\newcommand{\bbeta}{\boldsymbol{\eta}}

\newcommand{\bvarepsilon}{\boldsymbol{\varepsilon}}

\newcommand{\diag}{\operatorname{diag}}

\newcommand{\sign}{\operatorname{sign}}

\newcommand{\one}{{\mathbbm{1}}}

 				\newcommand{\mtilde}{~}   
 	 			\newcommand{\covop}{{\ccC}}

\declaretheoremstyle[
spaceabove=0.5cm, 
    	spacebelow=0.5cm, 
notefont=\itshape, notebraces={(}{)},
bodyfont=\normalfont,
headformat=\NAME~\NUMBER \NOTE
]{myremark}

\declaretheoremstyle[
    spaceabove=0.5cm, 
    spacebelow=0.5cm, 
    headfont=\normalfont\bfseries,
    notefont=\normalfont\bfseries, 
    notebraces={(}{)}, 
    bodyfont=\normalfont\itshape, 
    headpunct={.}
]{mytheorem}

\declaretheoremstyle[
    spaceabove=0.5cm, 
    spacebelow=0.5cm, 
    headfont=\normalfont\bfseries,
    notefont=\normalfont\bfseries, 
    notebraces={}{}, 
    bodyfont=\normalfont, 
    headpunct={.} 
]{myassumption} 

 \declaretheoremstyle[
    spaceabove=0.5cm, 
    spacebelow=0.5cm, 
    headfont=\normalfont\bfseries,
    notefont=\normalfont\bfseries, 
    notebraces={}{}, 
    bodyfont=\normalfont, 
    headpunct={.}
]{mydefinition} 
 
\theoremstyle{mytheorem}
\newtheorem{theorem}{Theorem}[section]
\newtheorem{proposition}[theorem]{Proposition}
\newtheorem{corollary}[theorem]{Corollary}
\newtheorem{lem}[theorem]{Lemma}

\theoremstyle{mydefinition}
\newtheorem{definition}[theorem]{Definition} 

\theoremstyle{myassumption}
\newtheorem{assumption}[theorem]{Assumption} 
\theoremstyle{myassumption}
\newtheorem*{assumption*}{Assumption}  
\theoremstyle{myassumption}
  
\theoremstyle{myassumption}
\newtheorem*{notations*}{Notation}  
 
\theoremstyle{myremark}
\newtheorem{remark}[theorem]{Remark}

\makeindex 
\numberwithin{equation}{section}

\colorlet{linkcolour}{cyan!40!blue}
\colorlet{urlcolour}{magenta}
\hypersetup{colorlinks=true, linkcolor=linkcolour, citecolor=linkcolour, urlcolor=urlcolour}

\setdescription{font=\normalfont\itshape}  
 		\allowdisplaybreaks

\title{{\Large \textsc{Detecting changes in Hilbert space data based on ``repeated'' and change-aligned principal components}}}
\date{}
\author{L. Torgovitski\footnote{Mathematical Institute, University of Cologne, Weyertal 86-90, 50931, Cologne, Germany. \\ This research was partially supported by the Friedrich Ebert Foundation, Germany.}}
\microtypesetup{protrusion=true}	
\begin{document}  
\maketitle

{\small
\begin{abstract}
We study a CUSUM (cumulative sums) procedure for the detection of changes in the means of weakly dependent time series within an abstract Hilbert space framework. We use an empirical projection approach via a principal component representation of the data, i.e., we work with the eigenelements of the (long run) covariance operator.

This article contributes to the existing theory in two directions: By means of a recent result of \citet{reimherr} we show, on one hand, that the commonly assumed \textit{separation of the leading eigenvalues} for CUSUM procedures can be avoided. 
This assumption is \textit{not} a consequence of the methodology but merely a consequence of the usual proof techniques.
On the other hand, we propose to consider \textit{change-aligned} principal components that allow to further reduce common assumptions on the eigenstructure under the alternative. 
This approach extends directly to \textit{multidirectional} changes, i.e. changes that occur at different time points and in different directions, by fusing sufficient information on them into the first component. The latter findings are illustrated by a few simulations and compared with existing procedures in a functional data framework.
\end{abstract}
} 
\newpage
\section{Introduction}
Over the last ten years, the analysis of change point problems in a \textit{functional data} framework became a quite popular area of research. The incorporation of projections on functional principal components, i.e. eigenfunctions of the (long run) covariance operator, 
is meanwhile a well-established technique for change point procedures and there is a considerable amount of literature on this topic such as, e.g., \citet{berkes2009functional}, \citet{hoermann2010weakly}, \citet{kirch2012functional}, \citet{horvath2013testing}, \citet{lukasz2015} and the book of \citet{horvath2012fda}, just to name a few.

As pointed out in \citet{reimherr} most literature uses the assumption of well-separated leading $d$ eigenvalues
\begin{equation}\label{eq:separationass}
	\lambda_1>\lambda_2>\ldots>\lambda_d>\lambda_{d+1}
\end{equation}
to ensure suitable estimation of principal components \textit{up to signs}.
In the latter article it is shown in a rather broad generality that for the estimation of subspaces spanned by the $d$ leading principal components only the gap $\lambda_d>\lambda_{d+1}$ is relevant, if these subspaces are estimated \textit{as a whole} and not in each direction separately. We use the work of \citet{reimherr} to show \textit{how} the assumption on distinct eigenvalues can be omitted in a common change point framework.
The idea is to consider one projection of \textit{functional partial sums} on a $d$-dimensional subspace rather than $d$ partial sums of one-dimensional projections.

Another contribution is the study of change corrected principal components. Loosely speaking, the idea is to correct the leading empirical principal component with an estimate of the \textit{change direction} such that, on one hand, the component is asymptotically oriented in the change direction under the alternative but, on the other hand, still estimates the true principal component under the null hypothesis.  Note that the change direction represents the most informative subspace in our change point setting.
 A remarkable feature of this alignment is that it allows us to test for more complex \textit{multidirectional} changes in exactly the same fashion and under the same set of assumptions. This approach can be related to the study of \textit{one-dimensional} projections of \citet{kirch2014highdim} under high-dimensionality and also to corrections of covariance estimates which are common practice in multivariate change point analysis (cf., e.g., Remark 3 of \citet{antoch1997}).

We assume that $H$ is a real and separable infinite dimensional Hilbert space. The inner product is denoted by $\langle x,y\rangle_H$ and the corresponding norm by $\|x\|_H$ for $x,y\in H$. There should be no confusion, if we simply write $\langle x,y\rangle$ and $\|x\|$ without the subscript $H$.
The mean $EX$ of a Hilbert space random variable $X$ is defined generally as the element that fulfills $E\langle X,y\rangle= \langle EX, y\rangle$ for all $y\in H$ given that $E\|X\|<\infty$. Note that we show our results in an abstract setting having the space $L^2[0,1]$ of square-integrable functions in mind. 
\\
\\
The structure of this article is as follows. We begin with a description of the model and of the hypothesis tests in Section \ref{sec:model}. In Section \ref{sec:pca} we continue to explain the principal components approach and
then state our theoretical results on the CUSUM test based on repeated eigenvalues in Section \ref{sec:asymp}. The change adjusted principal components are introduced and studied in Section \ref{sec:aligned} and afterwards 
extended to more complex multiple-direction alternatives in Section \ref{sec:multipledirections}. The analysis is complemented by a small simulation study in Section \ref{sec:simulations}. All proofs are postponed to Section \ref{sec:proofs}.

\label{sec:asymp}
\section{Model and testing problem}\label{sec:model}
Let $\{\eta_i\}_{i\in\mZ}$ denote the observable Hilbert space valued time series in a signal plus noise model
\[
	\eta_i = m_i + \varepsilon_i
\]
with means $m_i=E\eta_i$ and a \textit{centered}, \textit{strictly stationary} Hilbert space noise $\varepsilon_i$ with $E\|\varepsilon_1\|^2<\infty$. 
Some weak dependence condition will be imposed further below. 
We assume that $m_i=g(i/n)\Delta$, $i=1,\ldots,n$, where $\Delta\in H$, $\|\Delta\|=1$, is the normalized direction of the means and $g$ is a function that describes their magnitude. We assume $g$ to be piecewise Lipschitz-continuous and our aim is to test the stability of the means $m_1,\ldots,m_n$, i.e., the null hypothesis
\[
	H_0: ~g ~~\text{is a constant function on $[0,1]$}
\]
against a trend under the alternative specified by
\[
	H_A: ~g ~~\text{is a \textit{non}-constant function on $[0,1]$}.
\]
We have the classical \textit{abrupt} or \textit{epidemic} settings in mind but consider this broad alternative to underpin the generality of our results. Note that this model has been considered, e.g., by \citet{horvath2013testing} in a related functional framework. They focused more on other alternatives but also state or indicate some results in our situation.

To test $H_0$ versus $H_A$ we will use a CUSUM procedure relying on dimension reduction where the basis of the reduced subspaces will be generated by \textit{generalized} Hilbert space principal components. 
Those have been suggested in a functional two-sample context by \citet{reeder2012twoB} and are motivated by the optimality properties of principal components and of their empirical counterparts.
To state our results and our interpretation of the projection based CUSUM we need to introduce the generalized principal components first.

\section{Generalized principal components}\label{sec:pca}
The \textit{generalized} principal components are eigenelements of the long run covariance operator
\[		
	\ccC = \sum_{h\in\mZ} \ccC_h, \qquad  \ccC_h=E[\varepsilon_0 \otimes \varepsilon _h],
\]
where the \textit{tensor operator} $x \otimes y $ is defined via $(x \otimes y) z = \langle y, z \rangle x$ for any $x,y,z\in H$ and thus is linear in all components. Let $\|\cdot\|_{\cS}$ denote the \textit{Hilbert-Schmidt} norm. In the following we tacitly assume the summability
of the lagged covariance operators 
\begin{equation}\label{eq:summability}
\sum_{h\in\mZ}\|\ccC_h\|_{\cS}<\infty
\end{equation}
which ensures that $\ccC$ is a well-defined positive self-adjoint Hilbert-Schmidt operator. Let $\hat{\ccC}$ be for the moment a generic estimate of $\ccC$ and assume that $\hat{\ccC}$ is 
a self-adjoint Hilbert-Schmidt operator, too.
In this case we may represent both operators using the spectral decomposition
\begin{align}
	\ccC &= \sum_{j\in\mN} \lambda_j (v_j \otimes v_j),\qquad \hat{\ccC} = \sum_{j\in\mN} \hat{\lambda}_j (\hat{v}_j \otimes \hat{v}_j)
\end{align}
with real eigenvalues $\lambda_j$ and $\hat{\lambda}_j$. Without loss of generality we may assume $\lambda_1\geq \lambda_2 \geq \lambda_2\geq \ldots \geq 0$ and also that the $v_j$'s form an orthonormal basis of $H$. We make exactly the same assumption on the eigenelements of $\hat{\ccC}$ but do not assume that all $\hat{\lambda}_j$ are non-negative.

Note that we use the term \textit{generalized} to distinguish between eigenelements of $\ccC$ and eigenelements of the covariance operator $\ccC_0$. The latter will be called \textit{standard} principal components.

\section{Asymptotic behavior of the CUSUM statistic}\label{sec:asymp}
We want to consider the following projection based CUSUM statistics
\begin{equation}\label{eq:projectedrepresentation}
	\ccT^{(d)}=\max_{1\leq k <n}|\Sigma^{-1/2}S_k(\bbeta)|, \qquad \hat{\ccT}^{(d)}=\max_{1\leq k <n}|\hat{\Sigma}^{-1/2}S_k(\hat{\bbeta})|,
\end{equation}
where $\hat{\ccT}^{(d)}$ is the empirical counterpart of $\ccT^{(d)}$. The detectors \eqref{eq:projectedrepresentation} are based on projected time series
\begin{align}
\bbeta_i&=[\langle \eta_i, v_1 \rangle, \ldots, \langle \eta_i, v_d \rangle]^{\prime},\qquad \hat{\bbeta}_i=[\langle \eta_i, \hat{v}_1 \rangle, \ldots, \langle \eta_i, \hat{v}_d \rangle]^{\prime}
\end{align}
and we write $S_k(\bbeta)=\sum_{i=1}^k (\bbeta_i -\bar{\bbeta}_n)/n^{1/2}$ to denote their centered partial sums. Note that $\ccT^{(d)}$ is not fully observable in practice, since the eigenelements of $\ccC$ will rarely be known.
The matrix $\Sigma=\diag(\lambda_1,\ldots,\lambda_d)$ is the long run covariance matrix of $\{\bbeta_i\}_{i\in\mZ}$ and $\hat{\Sigma}=\diag(|\hat{\lambda}_1|,\ldots,|\hat{\lambda}_d|)$ is its generic estimate. (The diagonality of $\Sigma$ is a property of the generalized principal components.) 
Finally,  $|x|$ denotes the \textit{Euclidean} norm and thus $|x|_{\Sigma}=|\Sigma^{-1/2} x|$, $x\in\mR^{d}$, is also a norm whenever $\Sigma$ has full rank. 
To ensure this regularity we tacitly assume that $\lambda_d>0$. Later on we will ensure that $\lim_{n\rightarrow\infty}P(\hat{\lambda}_d>0)=1$ holds true under the null hypothesis, too. We do not require the latter under the alternative, but set $\hat{\ccT}^{(d)}:=\infty$ if $\hat{\lambda}_j=0$ for some $1\leq j\leq d$.

It is well-known that a \textit{functional central limit theorem} for the projected time series $\{\bbeta_i\}_{i\in\mZ}$ allows to establish weak convergence of $\ccT^{(d)}$ under the null hypothesis by means of the continuous mapping theorem. 
Under short range dependence (such as, e.g., strong mixing or $\ccL^p$-$m$-approximability) the statistics $\ccT^{(d)}$, $\hat{\ccT}^{(d)}$ and other related detectors are studied in, e.g., \citet{berkes2009functional}, \citet{hoermann2010weakly}, \citet{kirch2012functional} or \citet{horvath2013testing}. Under $H_0$ the following limit holds true
\begin{equation} \label{eq:contmap}
	\ccT^{(d)}\stackrel{\ccD}{\longrightarrow} \sup_{0\leq x\leq 1}\Big(\sum_{1\leq r \leq d} \ccB^2_r(x)\Big)^{1/2},
\end{equation}
as $n\rightarrow\infty$, where $\{\ccB_r(x), x\in[0,1]\}$ are independent standard Brownian bridges. Clearly, to establish the same convergence for $\hat{\ccT}^{(d)}$ we just have to replace the eigenelements of $\ccC$ with those of $\hat{\ccC}$. As pointed
out by \citet{reimherr} this is usually done by considering the estimation of the eigenstructure in each direction separately and not at once. This separated treatment is the origin of some technical issues and in particular has the drawback of the separation assumption \eqref{eq:separationass} on the eigenvalues.

We show how one may use the results of \citet{reimherr} to get rid of this assumption and impose only the following condition on the \textit{last} eigenvalue. 
 
\begin{assumption}\label{ass:disjoint} It holds that $\lambda_d>\lambda_{d+1}$.
\end{assumption}
 
The key to our results is the following CUSUM representation
\begin{equation}\label{eq:funcrepresent}
	{\ccT}^{(d)}=\max_{1\leq k <n}\|\ccC^{(d)}S_k(\eta)\|, \qquad \hat{\ccT}^{(d)}=\max_{1\leq k <n}\|\hat{\ccC}^{(d)}S_k(\eta)\|
\end{equation} 
which is based on \textit{truncated} operators defined by
\begin{align}\label{eq:roots}
\ccC^{(d)}= \sum_{j=1}^d \lambda_j^{-1/2} (v_j \otimes v_j), \qquad \hat{\ccC}^{(d)}= \sum_{j=1}^d |\hat{\lambda}_j|^{-1/2} (\hat{v}_j \otimes \hat{v}_j),
\end{align}
for any $d\in\mN$. Note that \eqref{eq:funcrepresent} just replaces the long run covariance matrices and the Euclidean norm in \eqref{eq:projectedrepresentation} with the Hilbert space analogues. The equalities in \eqref{eq:funcrepresent} are valid in view of \textit{Parseval's} identity that implies
\begin{align*}
 	\|\ccC^{(d)}S_k(\eta)\|^2&=\sum_{r=1}^\infty |\langle \ccC^{(d)}S_k(\eta),v_r\rangle|^2\\
 	&= \sum_{r=1}^\infty \Big|\Big\langle \langle \sum_{j=1}^d \lambda_j^{-1/2} S_k(\eta),v_j \rangle v_j,v_r\Big\rangle\Big|^2\\
 	&= \sum_{j=1}^d |\lambda_j^{-1/2} \langle S_k(\eta),v_j \rangle|^2 =|{\Sigma^{-1/2}}S_k(\bbeta)|^2
\end{align*}
for all $1\leq k<n$. The next proposition is proven by (slightly) adapting the proofs of Lemmas 3.1 and 3.2 of \citet{reimherr} where \eqref{eq:roots} is considered with $\lambda_j^{-1}$ instead of $\lambda_j^{-1/2}$. 
\begin{proposition}\label{prop:bounds} Given that $\|\ccC - \hat{\ccC}\|_{\cS}=o_P(1)$ it holds that $\|\ccC^{(d)} - \hat{\ccC}^{(d)}\|_{\cS} = o_P(1)$ as $n\rightarrow\infty$.
 \end{proposition}
In the sequel we will use the following bound on the maxima of partial sums
\begin{equation}\label{eq:funcsums}
 \max_{1\leq k <n}\|S_k(\varepsilon)\| =\cO_P(1),
\end{equation}
as $n\rightarrow\infty$. We will discuss this bound in \autoref{rem:conds}, below. First, we state the asymptotics under the null hypothesis.
\begin{theorem}\label{thm:convH0}
Let $\|\ccC - \hat{\ccC}\|_\cS=o_P(1)$ and \autoref{ass:disjoint} hold true. Given \eqref{eq:contmap} and \eqref{eq:funcsums}, it holds under $H_0$ that, as $n\rightarrow\infty$
\begin{equation}\label{eq:asymptNull}
	\hat{\ccT}^{(d)}\stackrel{\ccD}{\longrightarrow} \sup_{0\leq x\leq 1}\Big(\sum_{1\leq r \leq d} \ccB^2_r(x)\Big)^{1/2}.
\end{equation}
\end{theorem}

A widely used approach to estimate $\ccC$ is via Bartlett-type estimates which are defined by
\[
	\hat{\ccC}_B=\sum_{r=-n}^{n}\ccK(r/h)\hat{\ccC}_r
\]
with $\hat{\covop}_r=\sum_{i=1}^{n-r}[\eta_{i}-\bar{\eta}_n]\otimes [\eta_{i+r}-\bar{\eta}_n]/n$ for $r\geq 0$ and symmetrically with $\hat{\covop}_r=\sum_{i=1}^{n+r}[\eta_{i-r}-\bar{\eta}_n]\otimes [\eta_{i}-\bar{\eta}_n]/n$ for $r<0$.
The \textit{bandwidth} $h=h_n\in\mN$ fulfills $h\rightarrow\infty$ as $n\rightarrow\infty$ but $h=o(n)$. The \textit{kernel} $\ccK$ is symmetric, i.e. $\ccK(x)=\ccK(-x)$ for $x\in\mR$, and piecewise continuous with $\ccK(0)=0$. Moreover, it is continuous at $x=0$, bounded by  $|\ccK(x)|\leq c$ for some constant $c>0$ and has a bounded support such that $\ccK(x)=0$ for $|x|>a$ for some $a>0$.

To be more specific about the conditions of \autoref{thm:convH0} and some of the subsequent theorems we consider the following weak dependence concept. 
 \begin{definition}
The time series $\{\varepsilon_i\}_{i\in\mZ}$ is \,\(\ccL^p\)-\-\(m\)-\-\textit{approximable} if the following conditions are all satisfied:
1) It holds that  $E\|\varepsilon_0\|^p<\infty$, for some $p\geq 1$.
2) It holds that $\varepsilon_i=f(\zeta_i, \zeta_{i-1},\zeta_{i-2}, \ldots)$ where $f: S^\infty \rightarrow H$ is a measurable mapping from some measurable space $S$ and where the innovations $\zeta_i$ are i.i.d.\ \,\(S\)-valued random elements. 
3) It holds that $\sum_{m=1}^\infty [E\|\varepsilon_0-\varepsilon_0^{(m)}\|^p]^{1/p}<\infty$ where $\varepsilon_i^{(m)}$ are $m$-dependent copies of $\varepsilon_i$ defined by $\smash{\varepsilon^{(m)}_i= f(\zeta_{i} ,\ldots , \zeta_{i-m+1}, \zeta^{i}_{i-m}, \zeta^{i}_{i-(m+1)}, \ldots)}$, using a family  $\{\zeta_r, \zeta_{i}^{j}, \mtilde i,j,r \in \mZ\}$ of i.i.d.\ random variables. 
\end{definition} 

\begin{remark}\label{rem:conds}
In \autoref{thm:convH0} we have to ensure that \eqref{eq:summability}, \eqref{eq:contmap}, $\|\ccC - \hat{\ccC}\|_\cS=o_P(1)$ and that \eqref{eq:funcsums} hold true. Let $\{\varepsilon_i\}_{i\in\mZ}$ be $\ccL^2$-$m$-approximable. The summability assumption \eqref{eq:summability} is shown, e.g., in \citet{lukasz2015} and \eqref{eq:contmap} follows from Theorem A.1 of \citet{Aue2009}. Using the above Bartlett-type estimates the $\|\ccC - \hat{\ccC}_B\|_\cS=o_P(1)$ bound follows similarly to \cite{reeder2012twoB} and rates can be obtained, e.g., via \citet{whipple2014}, \citet{berkes2015rice} or \citet{lukasz2015} under $\ccL^4$-$m$-approximability with (partly) some additional but mild assumptions. We refer also to \citet{panaretos2013} for another broad setting. 
The bound \eqref{eq:funcsums} follows from \citet{jirak2013} or from \citet{berkes2013rice}.

Note that the conditions $\|\ccC - \hat{\ccC}\|_\cS=o_P(1)$ and $\max_{1\leq k <n}\|S_k(\varepsilon)\| =\cO_P(1)$ are directly related to those mentioned in Remark 3.1 of \citet{kirch2012functional} but here on the operator and on the Hilbert space level.
\end{remark}

Next, we state the asymptotics under the alternative.

\begin{proposition}\label{thm:alternativeGeneral}
Assume that we use estimates $\hat{\ccC}$ such that $|\langle \hat{v}_k,\Delta\rangle|=c+o_P(1)$ holds true for some $c>0$ and that $\hat{\lambda}_k =o_P(n)$ is fulfilled for some $1\leq k\leq d$. Given \eqref{eq:funcsums}, it holds under $H_A$ that
\begin{equation}\label{eq:asymptAlt}
	\lim_{n\rightarrow\infty}P(\hat{\ccT}^{(d)}>t)=1
\end{equation}
for any threshold $t\in\mR$.
\end{proposition} 

Above assumptions can be ensured using Bartlett-type estimates, e.g., under the concept of $\ccL^p$-$m$-approximability. The next proposition is related to the statement in (3.14) of \citet{horvath2013testing}. As already mentioned, the \textit{trend-alternative} case is indicated in the latter article, too. Note that no additional assumptions on the eigenvalues are necessary.

\begin{proposition}\label{thm:alternative}
Assume that $\{\varepsilon_i\}_{i\in\mZ}$ are $\ccL^4$-$m$-approximable and that we use estimates $\hat{\ccC}_B$ with $\int_{-\infty}^\infty\ccK(x)dx\neq 0$. Then, \eqref{eq:asymptAlt} holds true under $H_A$ for any $d\in\mN$.
\end{proposition}

In the i.i.d.\ case it appears more natural to use the estimate $\hat{\ccC}_0$ for $\ccC=\ccC_0$ rather than $\hat{\ccC}_B$. However, a well-known issue is that \textit{not all} changes can be detected then. In particular the condition $|\langle \hat{v}_k,\Delta\rangle|=c+o_P(1)$, $c>0$, for some $1\leq k \leq d$, needs to be imposed and is quite difficult to verify. For example it is discussed in \citet{berkes2009functional} and \citet{kirch2012functional} that this condition is fulfilled for changes $\Delta$ that are not orthogonal to all $v_1,\ldots,v_d$ and also for \textit{sufficiently} large trends.

One remedy is to use the estimate $\hat{\ccC}_B$ instead of $\hat{\ccC}_0$ which is discussed in \citet{horvath2013testing}. It forces the leading eigenvector $\hat{v}_1$ to align with the change direction. We discuss another \textit{more direct} approach in the next subsection that relies on a correction of the first eigenvector. This correction is not restricted to $\hat{\ccC}_0$ and can be applied to eigenvectors of $\hat{\ccC}_B$, too.

\section{Change-aligned principal components}\label{sec:aligned}
We define the \textit{aligned} first principal component by
\begin{equation}\label{eq:def_uhat}
	\hat{v}_1^\prime=[\hat{v}_1/n^{\gamma}+ \hat{s}\hat{u}]/\|\hat{v}_1/n^{\gamma} + \hat{s}\hat{u}\|
\end{equation}
for some $\gamma\in(0,1/2)$, where $\hat{u}=S_{\bm\hat{k}}(\eta)/n^{1/2}$, $\hat{s}=\sign\langle \hat{v}_1,\hat{u}\rangle$, and where $\bm\hat{k}$ is chosen such that 
$\|S_{\bm\hat{k}}(\eta)\|=\max_{1\leq k <n}\|S_k(\eta)\|$. 
Under the null hypothesis and given \eqref{eq:funcsums} it behaves asymptotically like the empirical principal component $\hat{v}_1$ whereas under the alternative it aligns with the change direction of $\Delta$.
The speed of the alignment is controlled via the parameter $\gamma$. Clearly, a larger $\gamma$ slows down the convergence under the null but speeds it up under the alternative.
Using $\hat{v}_1^\prime$ we define the adjusted statistic
\[
	\hat{\ccT}^{(d)\prime}=\max_{1\leq k <n}|S_k(\hat{\bbeta}^\prime)|_{\hat{\Sigma}}
\]	
where $\hat{\bbeta}_i^\prime=[\langle \eta_i, \hat{v}_1^\prime \rangle, \langle \eta_i, \hat{v}_2 \rangle,\ldots, \langle \eta_i, \hat{v}_d \rangle]^{\prime}$, i.e. with $\hat{v}_1$ being replaced by $\hat{v}_1^\prime$ but leaving all other components and in particular \textit{all} the eigenvalue estimates unchanged. Note that representation \eqref{eq:funcrepresent} is not valid anymore for $\hat{\ccT}^{(d)\prime}$ since the orthogonality is generally violated but this is not an issue in our context. 
 
\begin{corollary}\label{cor:nullAligned} Under the assumptions of \autoref{thm:convH0} the asymptotics \eqref{eq:asymptNull} hold true for $\hat{\ccT}^{(d)\prime}$ under $H_0$.  
\end{corollary}
\begin{theorem}\label{thm:alternativeAligned}
Assume that we use estimates $\hat{\ccC}$ such that $\hat{\lambda}_1=o_P(n)$ holds true as $n\rightarrow\infty$. Given \eqref{eq:funcsums}, it holds under $H_A$ that
\[
	\lim_{n\rightarrow\infty}P(\hat{\ccT}^{(d)\prime}>t)=1
\]
for any threshold $t\in\mR$.
\end{theorem} 
The assumptions of \autoref{thm:alternativeAligned} are fulfilled in the i.i.d.\ case if we work with $\hat{\ccC}=\hat{\ccC}_0$ (cf., e.g., \citet{berkes2009functional}). Under $\ccL^4$-$m$-approximability, they follow from the proof of \autoref{thm:alternative} if we use $\hat{\ccC}=\hat{\ccC}_B$ with $\int_{-\infty}^\infty\ccK(x)dx\neq 0$.  

\section{Extension to multiple change directions}\label{sec:multipledirections}
Another interesting advantage of the change-aligned principal components is that we may test far more complex alternatives within the same set of assumptions as in \autoref{thm:alternativeAligned}. Consider the model
\begin{equation}\label{eq:multiple_direction}
 	m_i=g_1(i/n)\Delta_1+\ldots + g_\varrho(i/n)\Delta_\varrho
\end{equation}
where $\Delta_1,\ldots, \Delta_\varrho\in H$, $\varrho\in\mN$, are orthonormal directions and $g_i$ are piecewise Lipschitz continuous functions.
The general test is now $H_0$: all $g_l$ are constant on $[0,1]$ versus $H_A$: at least one $g_l$ is non-constant on $[0,1]$. 
Proceeding as under the \textit{one direction} framework one may restate \autoref{thm:alternativeAligned}.
Note that the proof essentially relies on the convergence  $\|\hat{u}\|\rightarrow^{P}\ccS>0$, as $n\rightarrow\infty$, where $\hat{u}$ is defined in \eqref{eq:def_uhat} and where in this general case
\[
	\ccS:=\sup_{0\leq x\leq 1}\|\cG_{g_1}(x)\Delta_1 + \ldots + \cG_{g_\varrho}(x)\Delta_\varrho\|^2=\sup_{0\leq x\leq 1}|\cG_{g_1}(x)|^2+\ldots + |\cG_{g_\varrho}(x)|^2.
\]
Here, $\cG_{g_l}$ denote $\cG$, as defined in \eqref{eq:integralsup}, but with respect to trends $g_l$. The equality holds due to
Parseval's identity and due to the orthonormality of the change directions. Furthermore, using Bartlett-type estimates $\hat{\ccC}_B$ we can repeat (with some notational effort) \autoref{diss:lem:coefficient_convergence} and \autoref{diss:theorem:intermediate_conv_alt_bothconcepts}, e.g., under $\ccL^4$-$m$-approximability, to show that all assumptions of the analogue of \autoref{thm:alternativeAligned} are fulfilled, too.

Loosely speaking, the information on all changes is \textit{condensed} into one direction. This result may even be of some interest in a multivariate framework.

\section{Simulations}\label{sec:simulations}
 
In this section we illustrate our findings and compare the test based on aligned principal components with the principal component approach of \citet{berkes2009functional} and of \citet{horvath2013testing} in a functional framework with $H=L^2[0,1]$.
For the simulations we use \href{https://www.r-project.org/}{R}. We consider independent square-integrable functional observations generated from sample paths of standard Brownian motions on $[0,1]$. Those are computed using the \href{https://cran.r-project.org/web/packages/e1071/index.html}{e1071-package} and
then converted into functional objects with the \href{https://cran.r-project.org/web/packages/fda/index.html}{{fda-package}} using a Fourier basis of $25$ basis functions which results in a rather smooth representation. 
For our simulations under the alternative we use piecewise linear trend functions 
 \[
 	g_{[\theta_1,\theta_2]}(x)=(x-\theta_1)(\theta_2-\theta_1)^{-1} \one_{\{\theta_1<x\leq\theta_2\}}(x) +  \one_{\{\theta_2<x\}}(x), \quad 0<\theta_1\leq \theta_2<1.
 \]
 The particular trend is not important for our discussion and epidemic or other alternatives could be chosen as well.
 First, we show the effects of the change-alignment by using a sample of $n=200$. We consider an abrupt change setting with $g(x)=\frac{1}{3}g_{[1/2,1/2]}(x)$ and the change direction $\Delta=(v_{10}+v_{11}+v_{12})/3^{1/2}$. 
 Some of these functional observations are illustrated in \autoref{fig:toy_example} which shows them before and after the change. 
 The change direction is chosen with a sole purpose: it is \textit{not} aligned with any of the leading \textit{true} population principal components of the data. 
    \begin{figure} 
\centering
\subfloat[][]{\includegraphics[width=7cm, height=5cm]{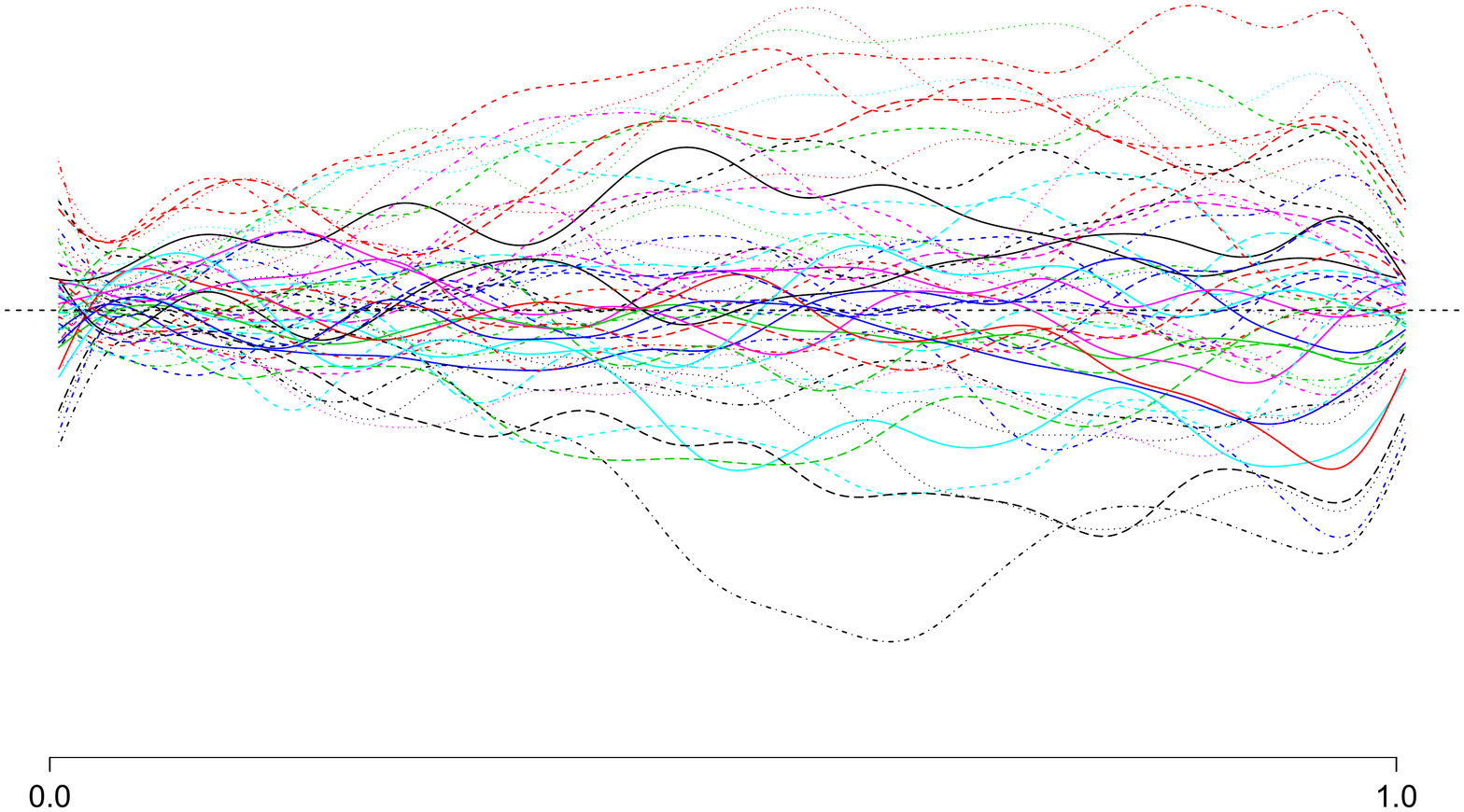}}
\subfloat[][]{\includegraphics[width=7cm, height=5cm]{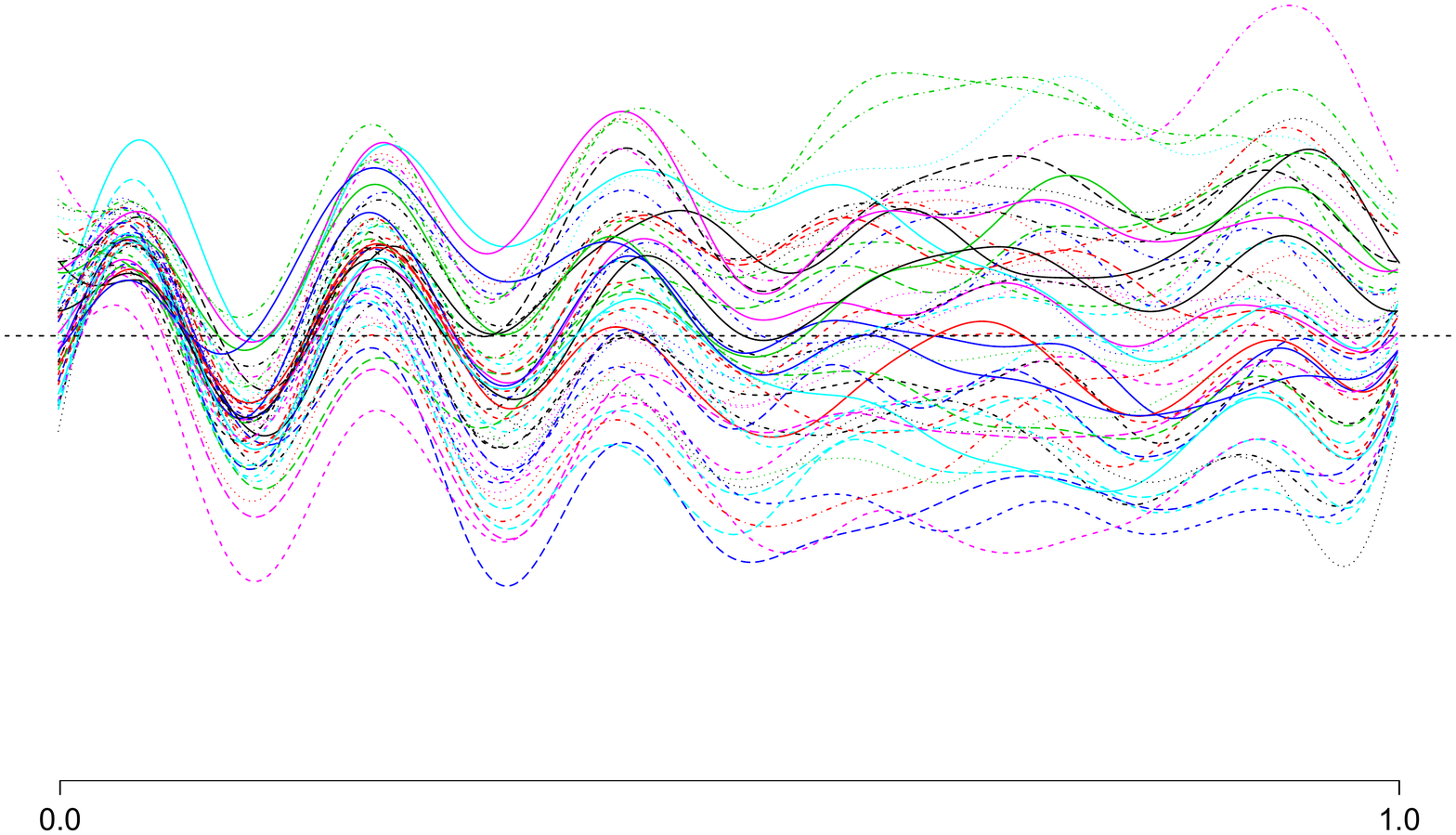}}  
\caption{Functional observations with and without a functional change.}
\label{fig:toy_example}
\end{figure}

 \begin{figure} 
\centering
\subfloat[][$\hat{v}_{1}^\prime$\\\\\mbox{\!\!\!\!aligned}]{%
\includegraphics[width=3cm, height=3cm]{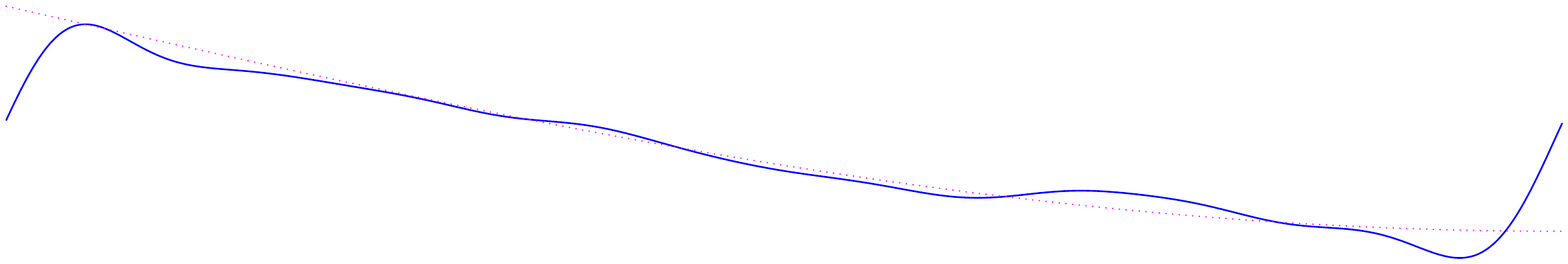}
}
\subfloat[][$\hat{v}_{1,B}^\prime$\\\\\mbox{\!\!\!\!aligned}]{%
\includegraphics[width=3cm, height=3cm]{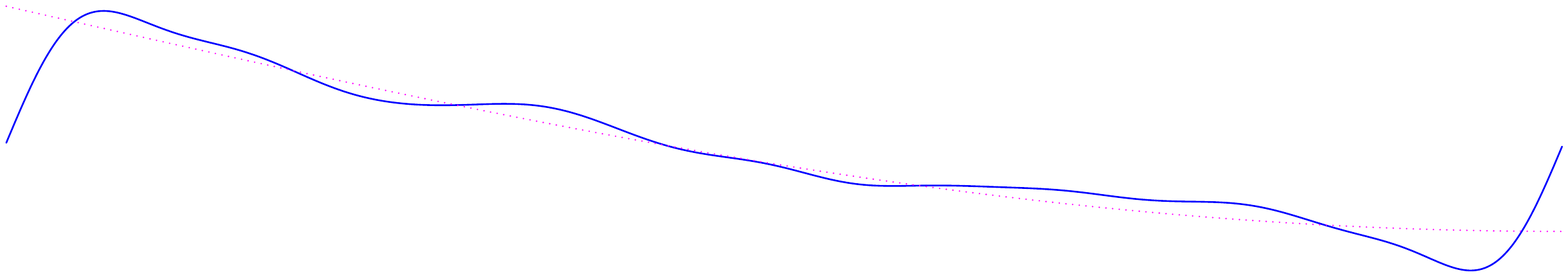}
}
\hspace{0.5cm}
\subfloat[][$\hat{v}_{1}^\prime$\\\\\mbox{\!\!\!\!aligned}]{%
\includegraphics[width=3cm, height=3cm]{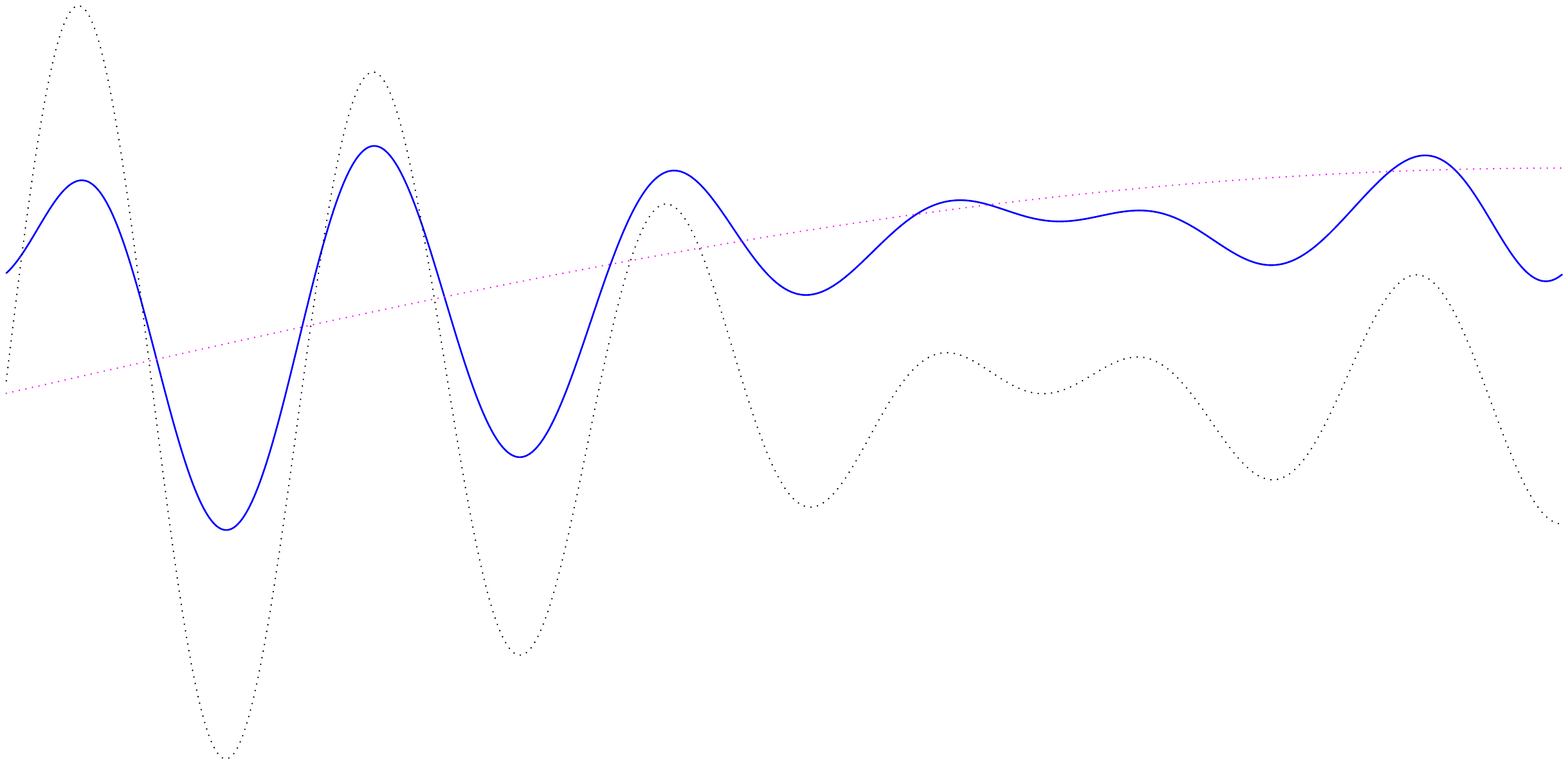}
}
\subfloat[][$\hat{v}_{1,B}^\prime$\\\\\mbox{\!\!\!\!aligned}]{%
\includegraphics[width=3cm, height=3cm]{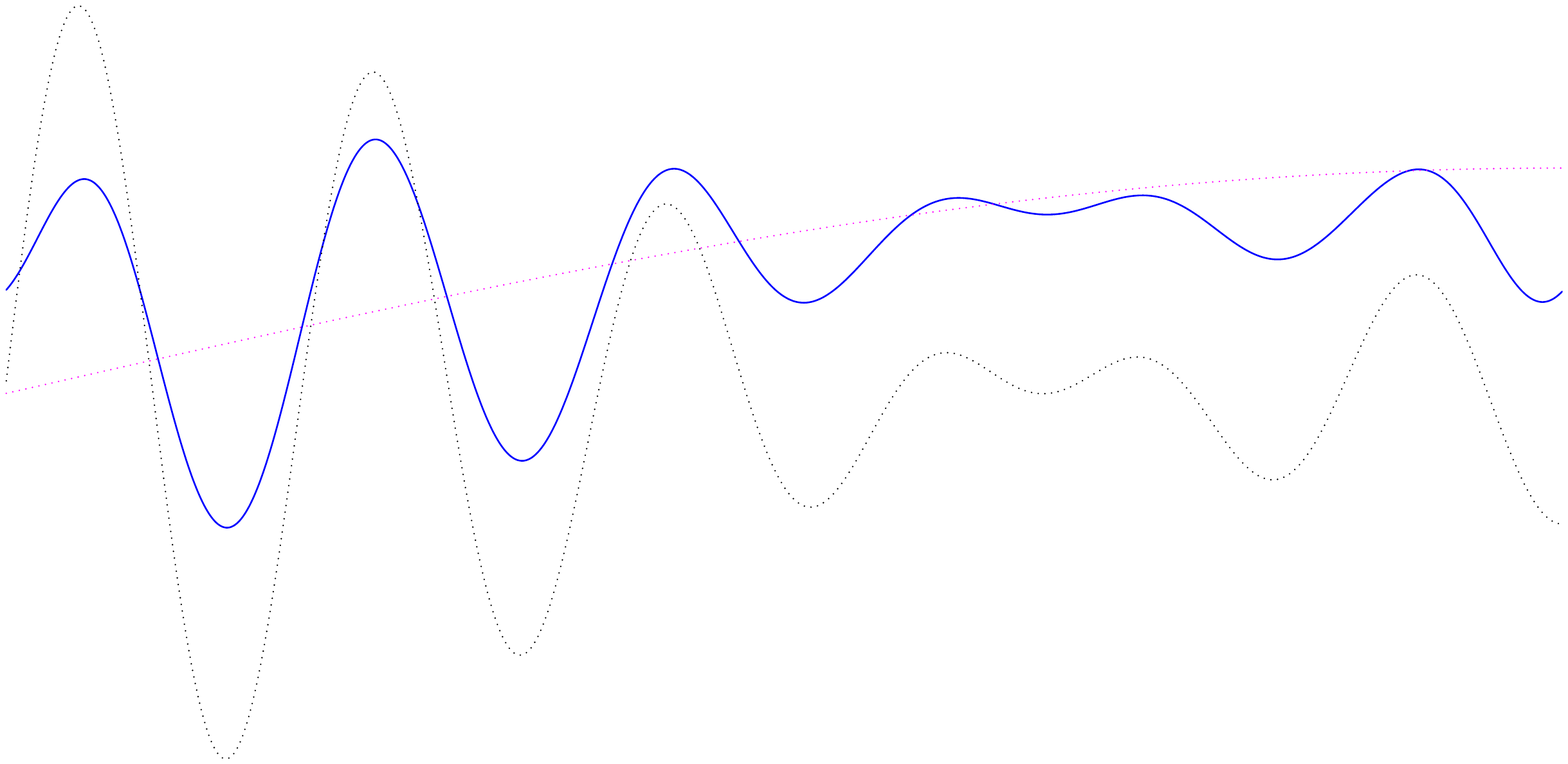}
}
\\ 
\subfloat[][$\hat{v}_{1}$]{%
\includegraphics[width=3cm, height=3cm]{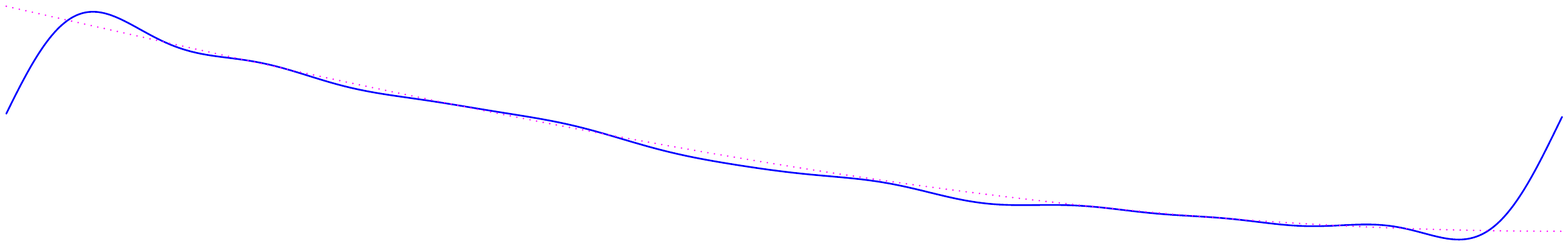}
}%
\subfloat[][$\hat{v}_{1,B}$]{%
\includegraphics[width=3cm, height=3cm]{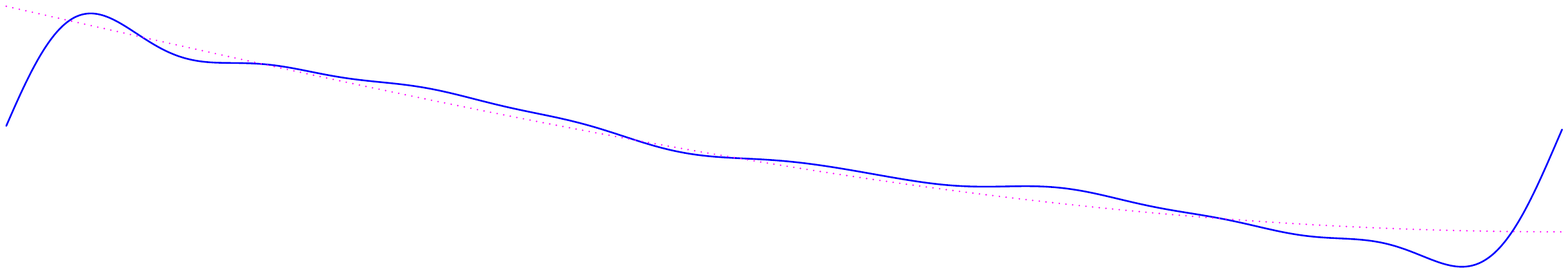}
}%
\hspace{0.5cm}
\subfloat[][$\hat{v}_{1}$]{%
\includegraphics[width=3cm, height=3cm]{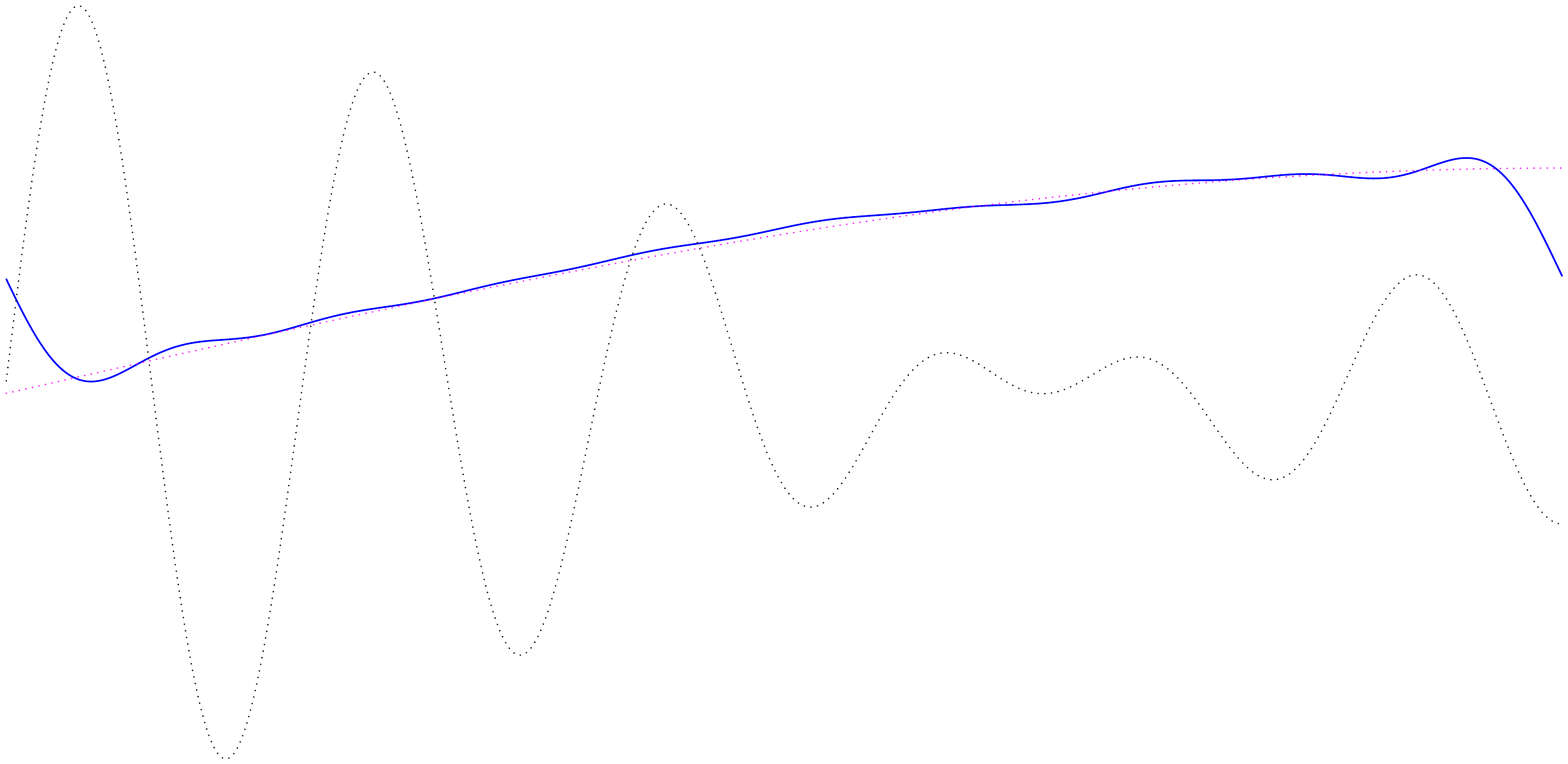}
}%
\subfloat[][$\hat{v}_{1,B}$]{%
\includegraphics[width=3cm, height=3cm]{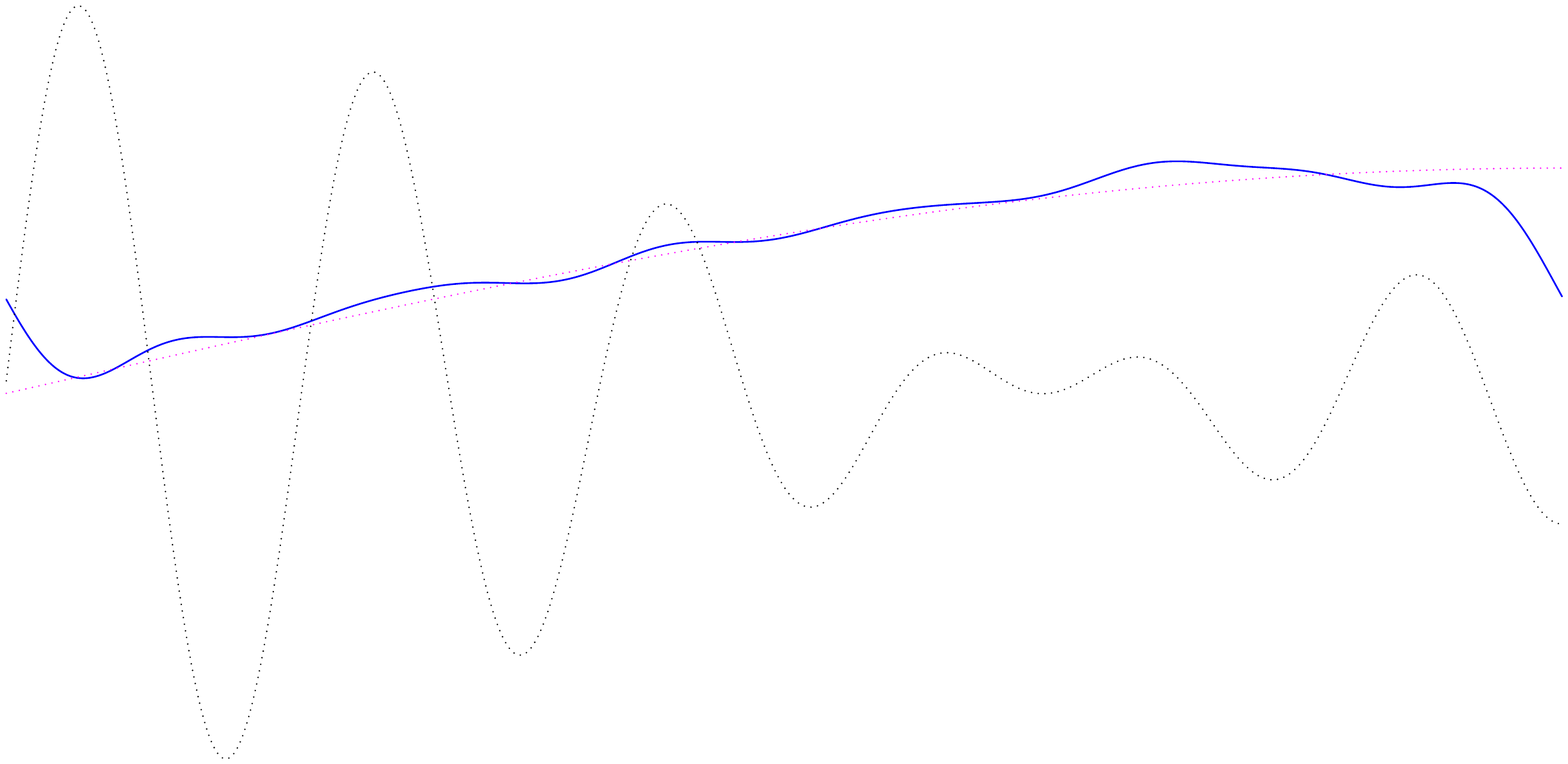}
}%
\\
\vspace{-0.2cm}
\subfloat[][$\hat{v}_{2}$]{%
\includegraphics[width=3cm, height=3cm]{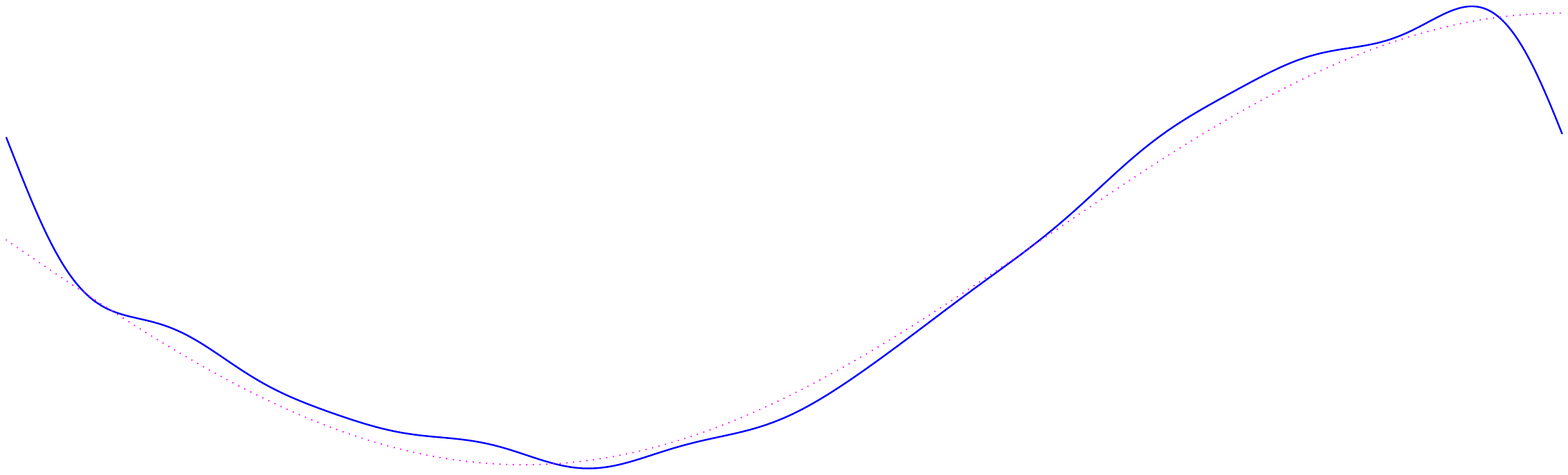}
}%
\subfloat[][$\hat{v}_{2,B}$]{%
\includegraphics[width=3cm, height=3cm]{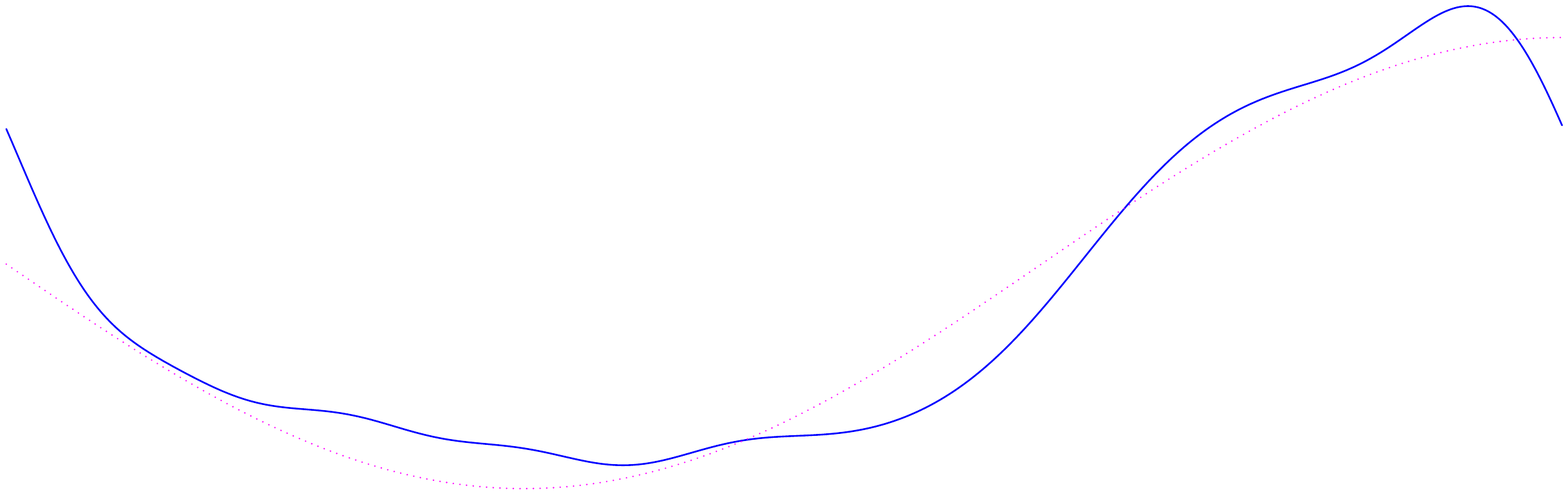}
}%
\hspace{0.5cm}
\subfloat[][$\hat{v}_{2}$]{%
\includegraphics[width=3cm, height=3cm]{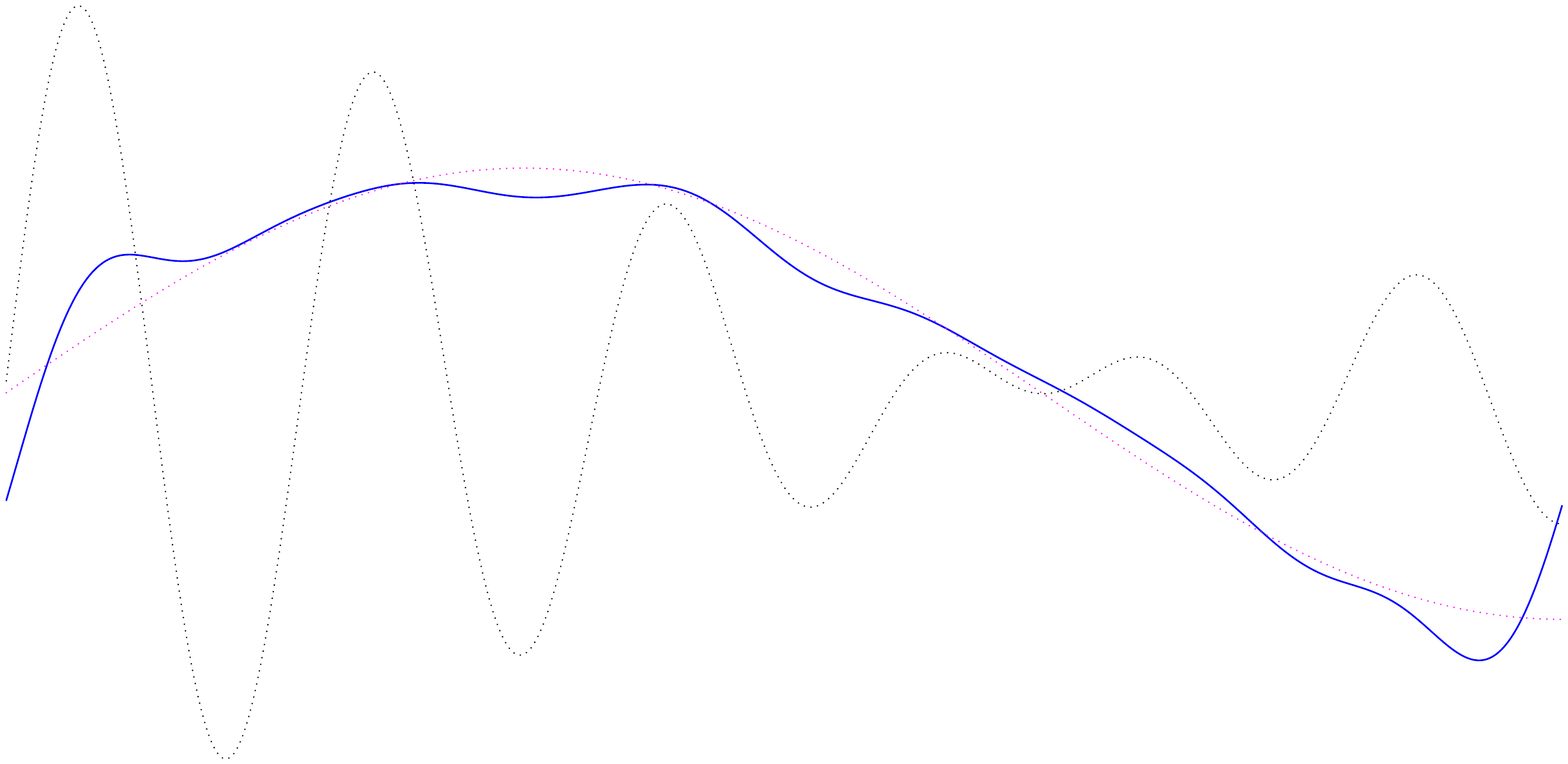}
}%
\subfloat[][$\hat{v}_{2,B}$]{%
\includegraphics[width=3cm, height=3cm]{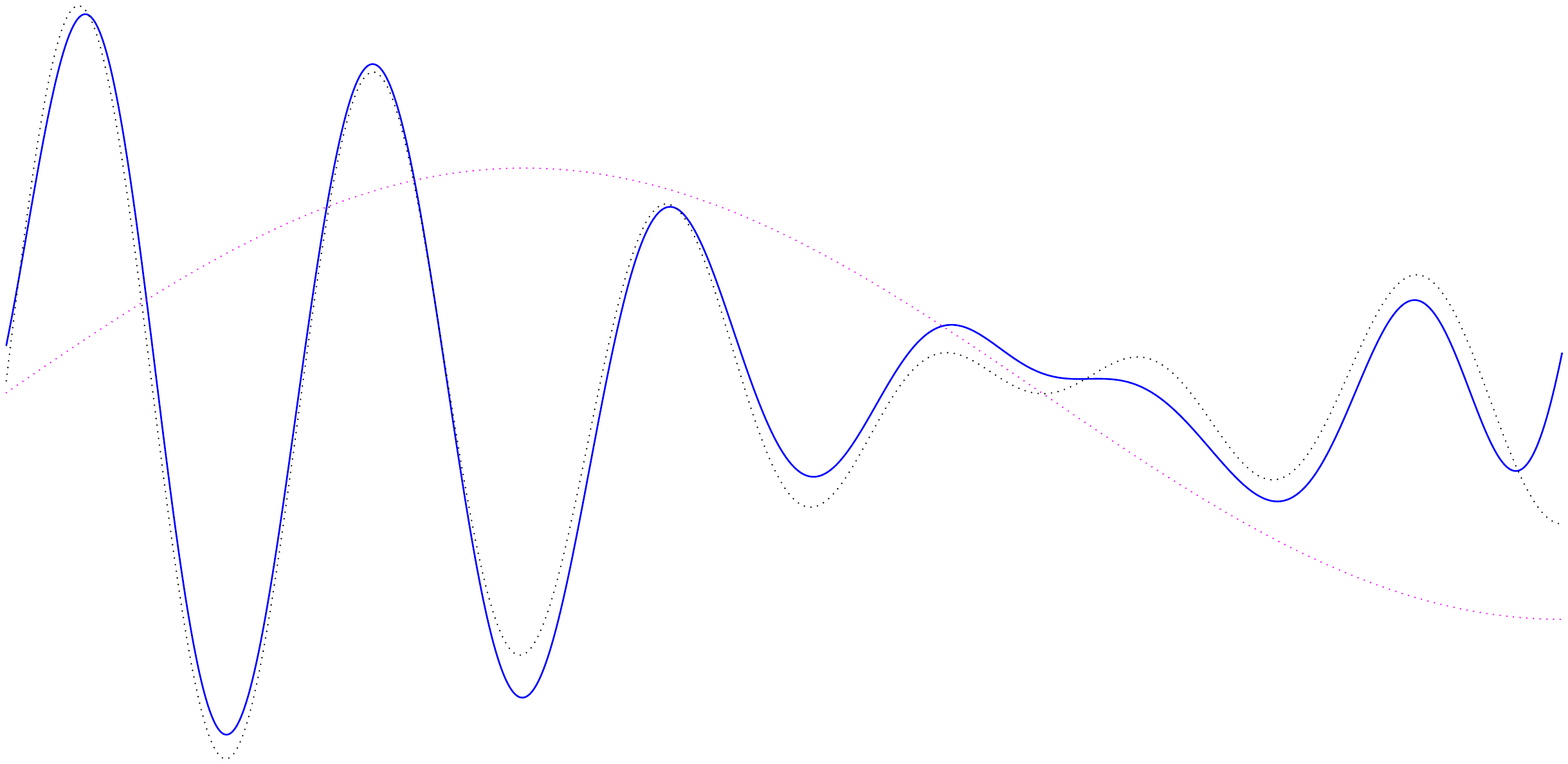}
}%
\\
\vspace{-0.2cm}
\subfloat[][$\hat{v}_{3}$]{%
\includegraphics[width=3cm, height=3cm]{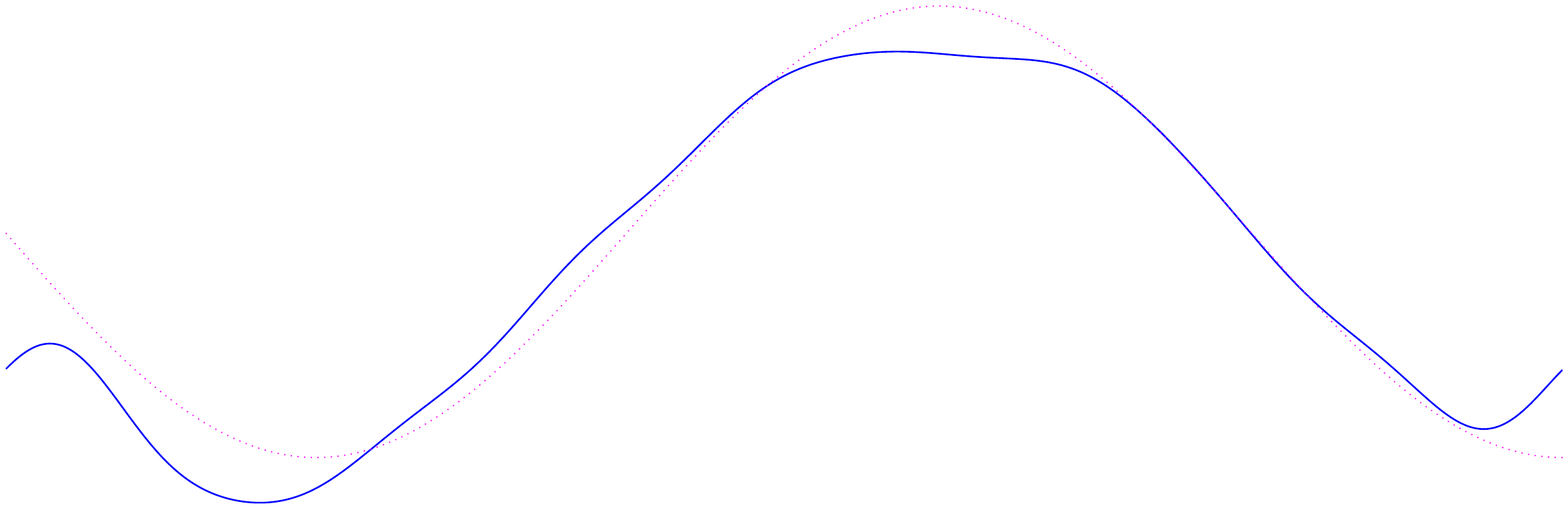}
}%
\subfloat[][$\hat{v}_{3,B}$]{%
\includegraphics[width=3cm, height=3cm]{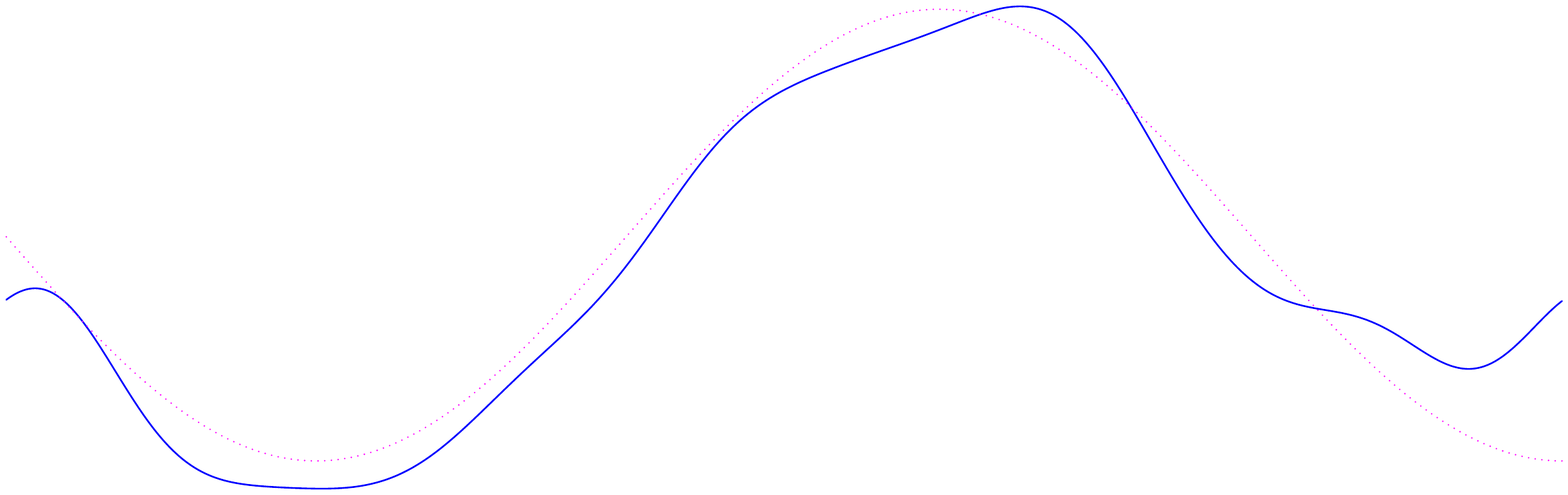}
}%
\hspace{0.5cm}
\subfloat[][$\hat{v}_{3}$]{%
\includegraphics[width=3cm, height=3cm]{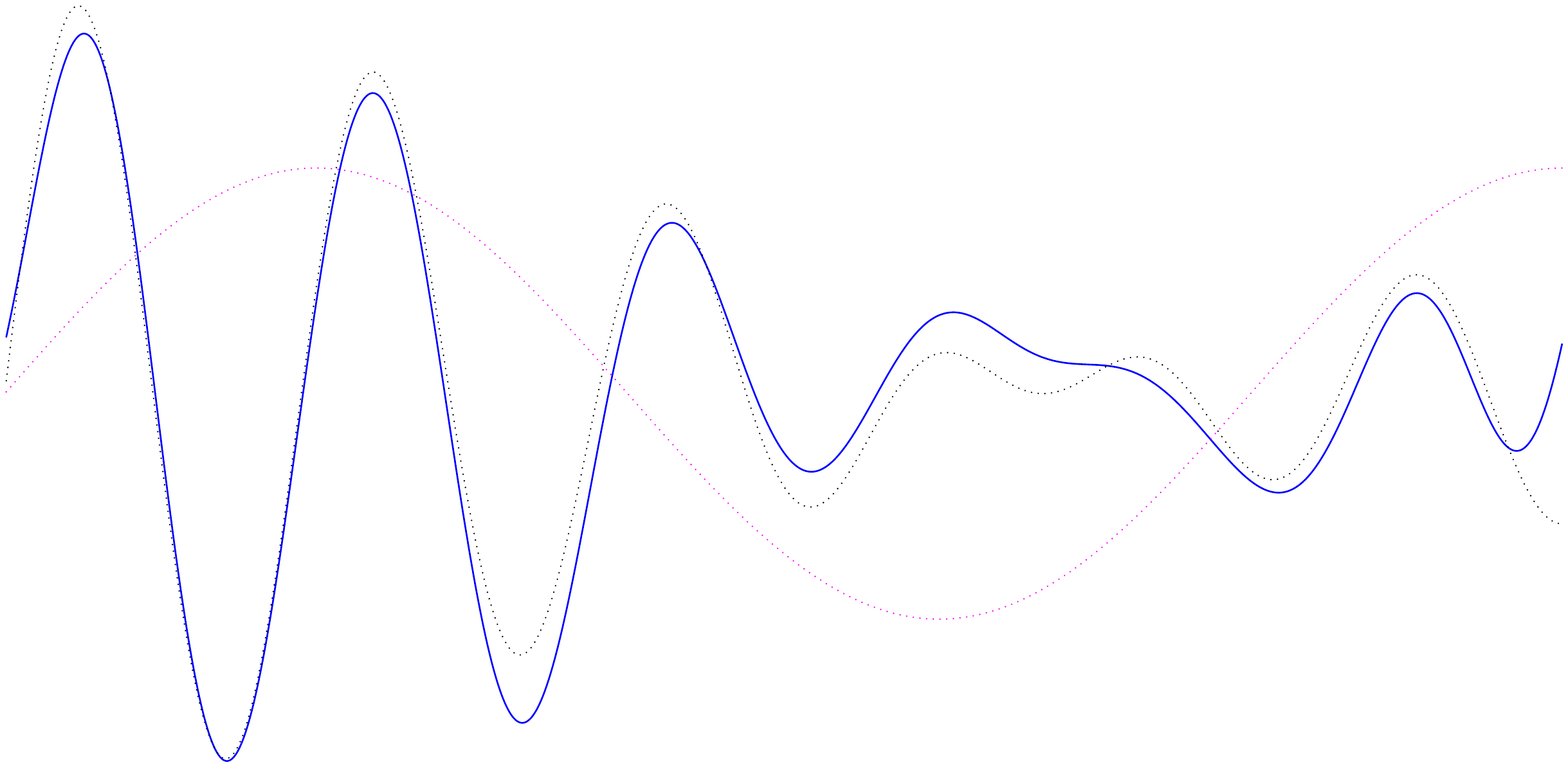}
}%
\subfloat[][$\hat{v}_{3,B}$]{%
\includegraphics[width=3cm, height=3cm]{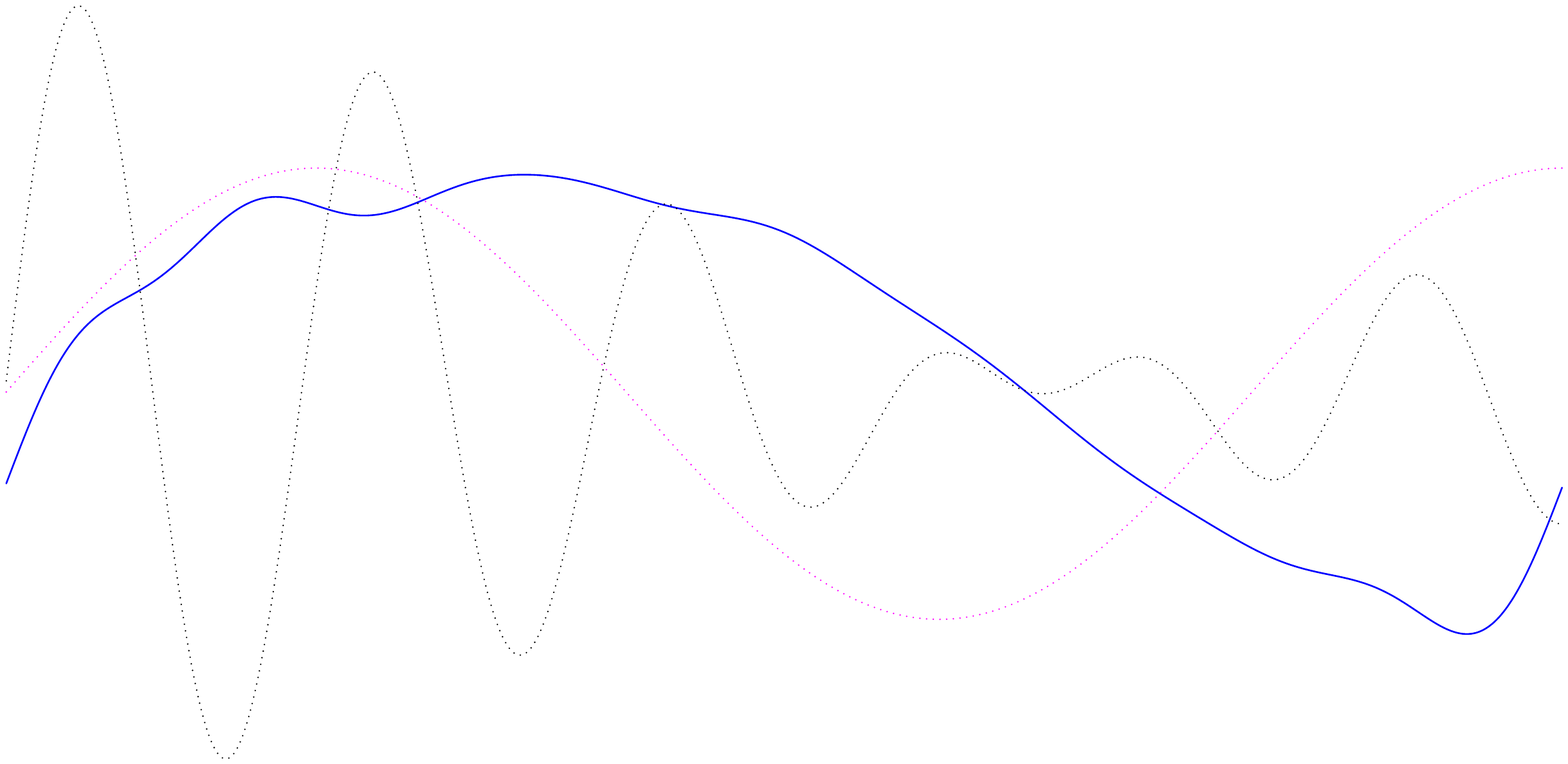}
}%
\\
\vspace{-0.2cm}
\subfloat[][$\hat{v}_{4}\\\\$\mbox{\!\!\!\!\!\!\!\!\!\!$H_0$ standard}]{%
\includegraphics[width=3cm, height=3cm]{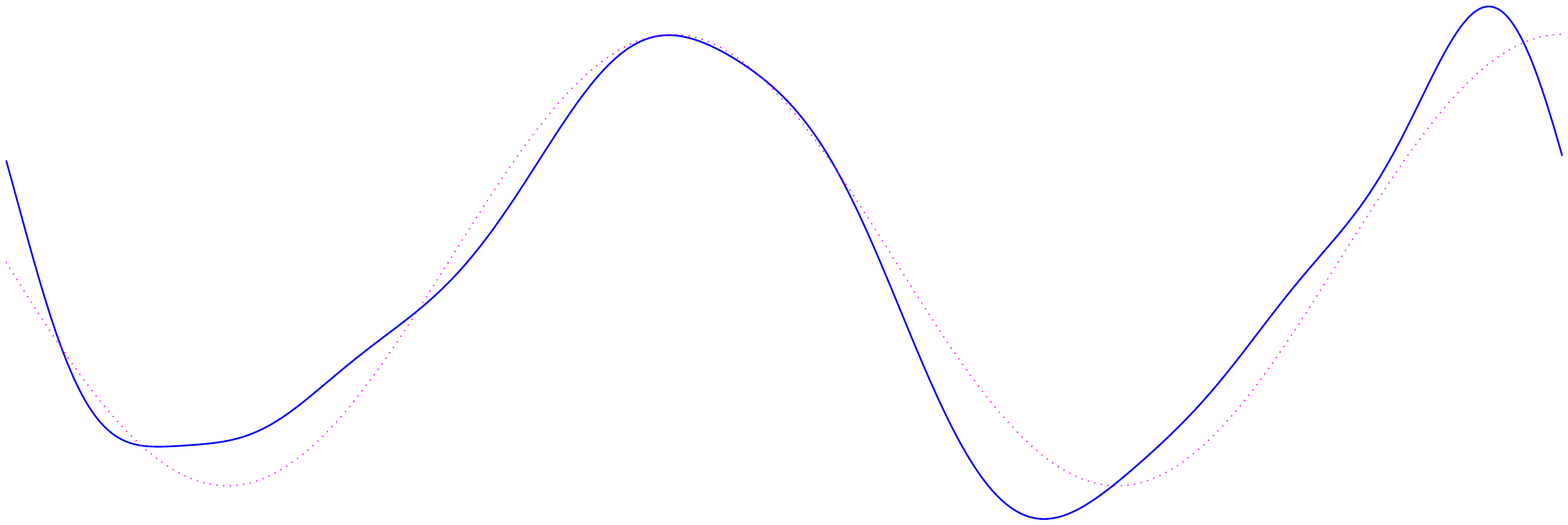}
}%
\subfloat[][$\hat{v}_{4,B}\\\\$\mbox{\!\!\!\!\!\!\!\!\!\!$H_0$ generalized}]{%
\includegraphics[width=3cm, height=3cm]{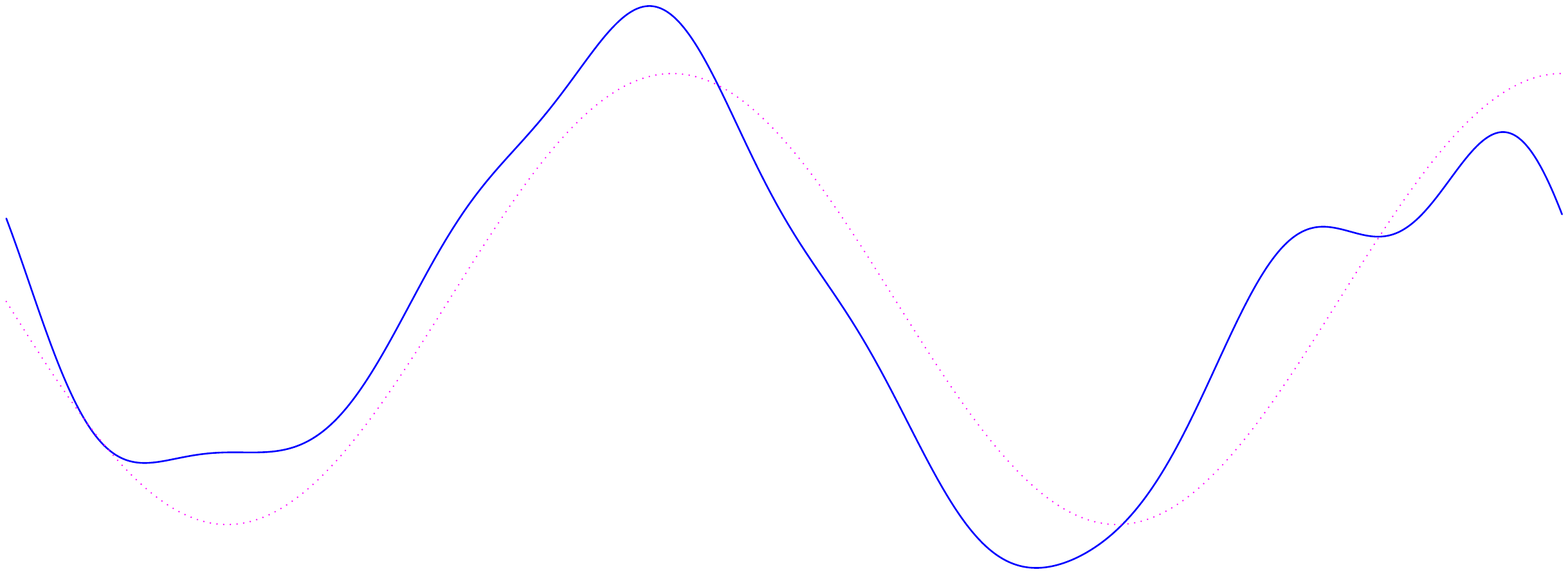}
}%
\hspace{0.5cm}
\subfloat[][$\hat{v}_{4}\\\\$\mbox{\!\!\!\!\!\!\!\!\!\!$H_A$ standard}]{%
\includegraphics[width=3cm, height=3cm]{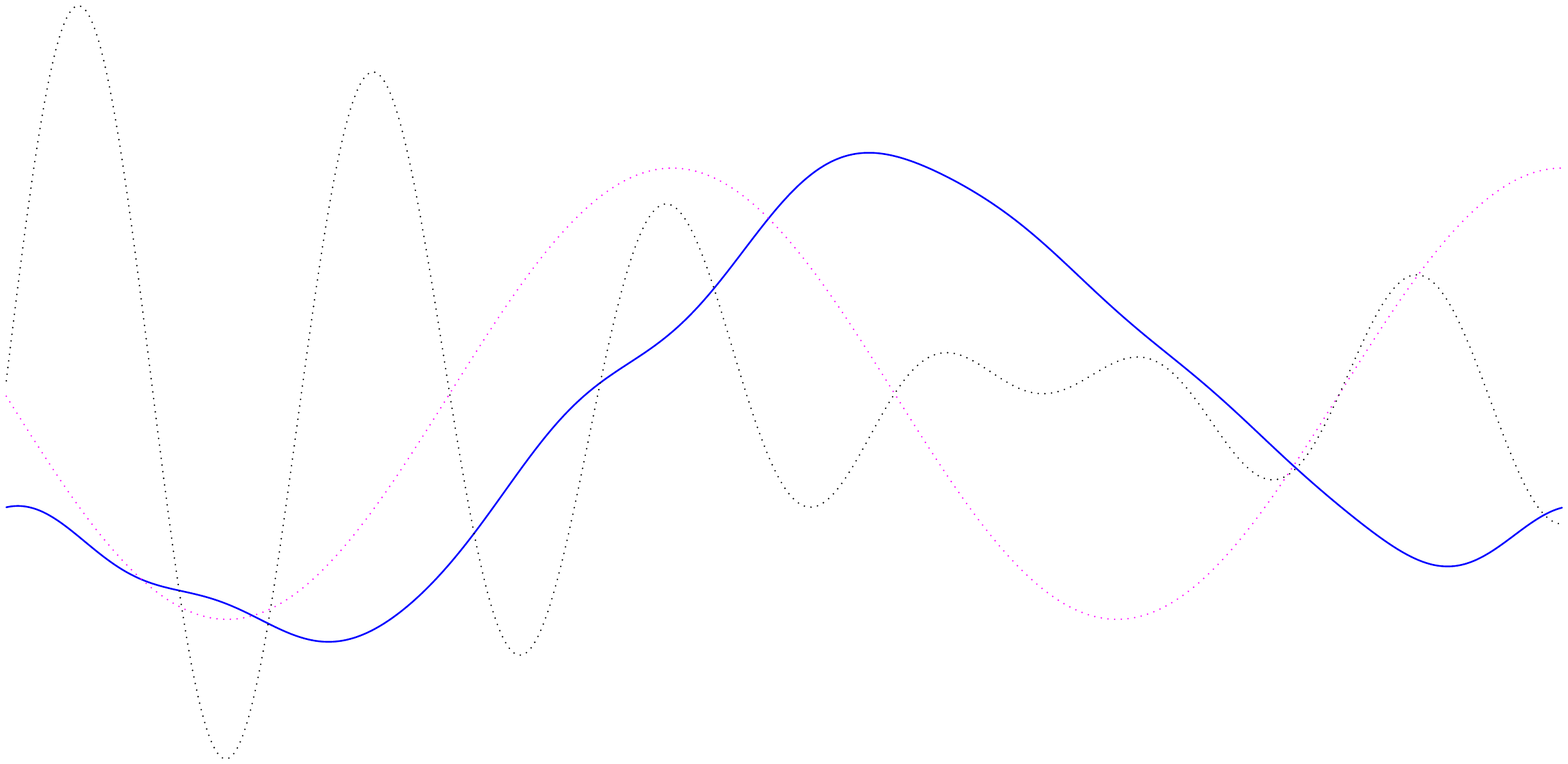}
}%
\subfloat[][$\hat{v}_{4,B}\\\\$\mbox{\!\!\!\!\!\!\!\!\!\!$H_A$ generalized}]{%
\includegraphics[width=3cm, height=3cm]{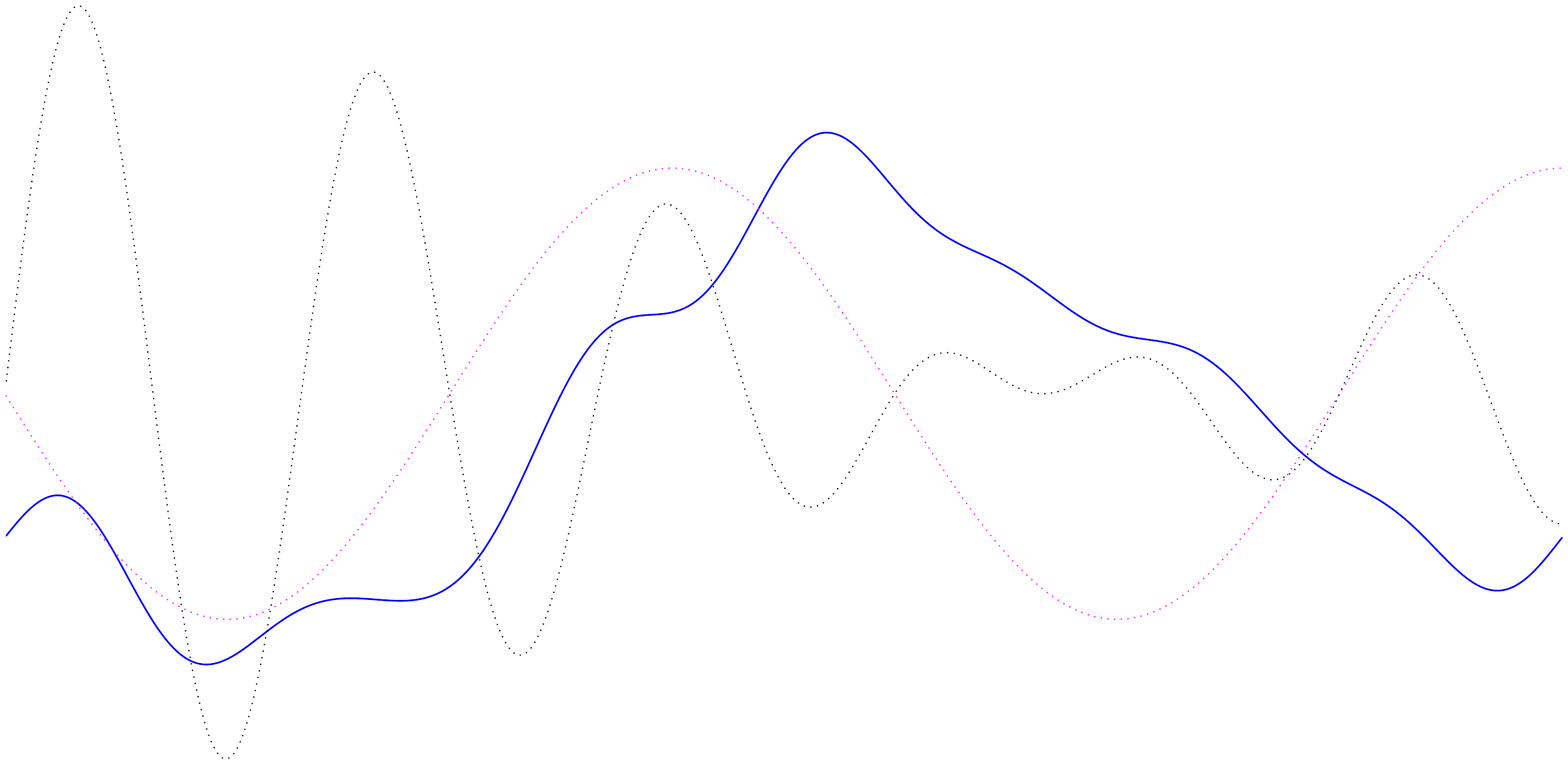}
}%
\caption{Empirical standard principal components $\hat{v}_1,\ldots,\hat{v}_4$ under $H_0$ and $H_A$ (1st and 2nd columns) and empirical generalized principal components $\hat{v}_{1,B},\ldots,\hat{v}_{4,B}$ under $H_0$ and $H_A$ (3rd and 4th columns) in ascending order. 
The dotted red lines indicate the true principal components $v_1,\ldots, v_4$ and the dotted black lines additionally the change direction $\Delta$ under the alternative. The first row shows the aligned 1st components  $\hat{v}_{1}^\prime$ and $\hat{v}_{1,B}^\prime$.}%
\label{fig:princip_comp}
\end{figure}

 The same data is used for \autoref{fig:princip_comp} which shows the four leading empirical standard (or generalized) principal components $\hat{v}_{j}$ and $\hat{v}_{j,B}$ and the corresponding change-aligned first components $\hat{v}_{1}^\prime$ and $\hat{v}_{1,B}^\prime$.
 Those are computed either with $\smash{\hat{\ccC}_0}$ or using a flat-top Bartlett estimate $\smash{\hat{\ccC}_B}$ with kernel $\smash{\ccK(x)=\one_{\{|x|\leq 1\}}(x)}$ and bandwidth $\smash{h=\lfloor n^{1/5}\rfloor}$. 
 For the change-aligned components we choose $\gamma=2/5$.
 In \autoref{fig:princip_comp} we see that that the alignment has only little effect under $H_0$ but a rather strong impact under $H_A$. 
 We also see that the first standard (and the first generalized) principal components are not capturing the large but quite oscillating change - even though the first direction \textit{explains} about $82\%$ $(73\%)$ of the variability in the \textit{change-contaminated} data.  
 This is substantially improved for the aligned counterparts  (in the first row) which move both, as expected, in the right direction. Note that the generalized empirical principal components would also push the change more into the first component and achieve a comparable effect under a larger sample size or a larger bandwidth.

 Next, we compare the performance of the test based on change-aligned components with the usual ones, i.e. with the approach considered in, e.g., \citet{berkes2009functional} and \citet{horvath2013testing}.
 The critical values are based on the limiting distribution \eqref{eq:asymptNull} and may be obtained, e.g., from Table 2 of \citet{kiefer1959}. The reported rejection rates in \autoref{tab:rejectionrates} are based on 1000 repetitions. 
 
 \begin{remark}
 Note that the aligned principal components ensure that the first projection contains \textit{asymptotically} all information on \textit{any} change $\Delta$ (and sufficient information on any changes $\Delta_1, \ldots, \Delta_\varrho$). 
 Thus, we propose to consider only this direction for testing and choose $d=1$ which has the additional advantage that no selection criteria for $d$ is necessary.  As already mentioned in the introduction, one-dimensional projections are also considered by \citet{kirch2014highdim}. The authors indicate in their \textit{Conclusion}-Section that (random) projections related to principal components might be of some further interest. 
 Our construction of $\hat{v}_1^\prime$ can be regarded as an explicit example of such a projection.
 \end{remark}
 \newpage
 In \autoref{tab:rejectionrates} we consider the following settings where we include $sin(t)$ and $t$ because both alternatives were considered in \citet{berkes2009functional}:
 \begin{enumerate}
 \item[A.]  $H_0$,
 \item[B.] $\Delta(t)=c\sin(t)$ and $g=\frac{1}{3}g_{[1/2,1/2]}$,
\item[C.] $\Delta(t)=v_{10}(t)$ and $g=\frac{1}{2}g_{[1/2,1/2]}$,
 \item[D.]  $\Delta(t) =c t$ and $g=\frac{1}{4}g_{[1/3,2/3]}$,
\item[E.] $\Delta(t) =c \cos(t)$ and $g=\frac{1}{3}g_{[1/3,2/3]}$,
\item[F.] $\Delta_1(t)=v_{10}(t)$, $\Delta_2(t)=v_{15}(t)$ and $g_1=\frac{1}{8^{1/2}}g_{[3/5,1]}$, $g_2=\frac{1}{8^{1/2}}g_{[1/3,2/3]}$.
 \end{enumerate}
 The constant $c$ is always chosen such that $\Delta$ is normalized. The scaling of the trend functions is always chosen such that the increase in the power is visible when the sample size $n$ increases. 
 We see in \autoref{tab:rejectionrates} that the size remains stable under alignment (column A) and that the power is not affected whenever it is already high without alignment (columns B, D and E) but may substantially increase if it is not (columns C and F).

\begin{table}
\centering
{\small
\begin{tabular}{rrrrrrrrrrrrr} 
  & $\hat{v}_1$  &  $\hat{v}_1^{\prime}$ & $\hat{v}_{1,B}$ &  $\hat{v}_{1,B}^{\prime}$  & $\hat{v}_1$  &  $\hat{v}_1^{\prime}$ & $\hat{v}_{1,B}$ &  $\hat{v}_{1,B}^{\prime}$  & $\hat{v}_1$  &  $\hat{v}_1^{\prime}$ & $\hat{v}_{1,B}$ &  $\hat{v}_{1,B}^{\prime}$  \\
  \cmidrule(r){2-5} \cmidrule(r){6-9}\cmidrule(r){10-13}
\\ 
$n$ & \multicolumn{4}{r}{\textbf{A}} & \multicolumn{4}{r}{B} & \multicolumn{4}{r}{\textbf{C}}\\
\cmidrule(r){1-1} \cmidrule(r){2-5} \cmidrule(r){6-9}\cmidrule(r){10-13} 

$100$  & \textbf{7.9}& \textbf{10.0} & \textbf{7.3} & \textbf{7.9} & 69.5& 70.6 &66.0 & 67.2& \textbf{11.6}& \textbf{96.9} & \textbf{64.7} & \textbf{96.4} \\
$200$	& \textbf{8.5}&\textbf{8.1} & \textbf{8.7} & \textbf{9.1} &95.1& 95.4 & 94.1 &94.4 &  \textbf{13.6}& \textbf{100} & \textbf{65.7} & \textbf{100}\\
$300$  & \textbf{9.5}& \textbf{9.9} & \textbf{9.4} & \textbf{10.5} &99.0  & 99.1 & 99.0 & 99.0&   1\textbf{3.6}& \textbf{100} & \textbf{92.7} & \textbf{100} \\
$400$ & \textbf{9.2}& \textbf{8.3} & \textbf{8.4}&  \textbf{8.9} & 99.8 & 99.8 & 99.0 & 99.0 &\textbf{14.3}& \textbf{100} & \textbf{95.5} & \textbf{100}\\
$500$ & \textbf{9.9}& \textbf{9.0} & \textbf{9.1} & \textbf{9.4} & 100 & 100 & 100 &100& \textbf{13.8} & \textbf{100} & \textbf{96.6} & \textbf{100}\\
\\
 
& \multicolumn{4}{r}{D} & \multicolumn{4}{r}{E} & \multicolumn{4}{r}{\textbf{F}}\\  
  \cmidrule(r){2-5} \cmidrule(r){6-9}\cmidrule(r){10-13}
$100$ &  41.3& 44.1& 40.6 & 42.0 & 52.8& 57.1 & 52.9 & 56.6 & \textbf{10.7}& \textbf{36.9} & \textbf{24.9} & \textbf{51.6}\\
$200$ &  70.0& 68.0 & 64.8 & 65.9 & 78.7& 86.5 & 80.6 & 87.4 &\textbf{9.5}& \textbf{87.4} & \textbf{24.3} & \textbf{90.1}\\
$300$ &  85.2& 87.3 & 85.7 & 86.3 & 91.7& 97.3 & 93.1 & 97.7 & \textbf{13.2}& \textbf{100}& \textbf{41.3} & \textbf{100}\\
$400$ &  94.3& 94.1 & 93.2 & 93.5 & 96.6 & 99.6 & 98.0 &99.6 & \textbf{9.6} & \textbf{100} &\textbf{42.4} & \textbf{100}\\
$500$ &  97.8& 98.2 & 97.3 & 97.5 & 99.1 & 99.8 & 99.6 & 99.8 & \textbf{11.6}& \textbf{100} & \textbf{40.8} & \textbf{100}\\
\end{tabular} 
}
\caption{Empirical rejection rates in percent. The critical value is chosen such that $10\%$ are rejected asymptotically under $H_0$.}
\label{tab:rejectionrates}
\end{table}

\section*{Conclusion}
From a theoretical point of view, we reduced the assumptions on the standard CUSUM procedure in a functional data setting under the null and under the alternative hypotheses in a flexible \textit{change in the mean} framework.
From a practical point of view, we proposed change-aligned principal components and demonstrated that they are useful modifications of the common principal component approach.
 
 \newpage
\section{Proofs}\label{sec:proofs}

\begin{proof}[Proof of \autoref{prop:bounds}] 
First, we define the \textit{mixed} operator
\[
	\hat{\ccC}^{(d)\prime}= \sum_{j=1}^d \lambda_j^{-1/2} (\hat{v}_j \otimes \hat{v}_j)
\]
and observe by the mean value theorem that $4(x^{-1/2}-y^{-1/2})^2 \leq (x-y)^2/(\min(x,y))^3$ for $x,y> 0$. Hence, slightly modifying the proofs of Lemmas 3.1 and 3.2 of \citet{reimherr}, we get
\[
	\|\ccC^{(d)} - \hat{\ccC}^{(d)\prime}\|^2_{\cS} \leq \frac{4d\|\ccC - \hat{\ccC}\|_{\cS}^2}{\lambda_{d}}\bigg(\frac{1}{(\lambda_{d}-\lambda_{d+1})^2}+ \frac{1}{(\hat{\lambda}_{d}-\hat{\lambda}_{d+1})^2}+\frac{1}{\lambda_{d}^2}\bigg)
\]
and
\[
	\|\hat{\ccC}^{(d)\prime} - \hat{\ccC}^{(d)}\|_{\cS}^2 \leq d \|\ccC - \hat{\ccC}\|_{\cS}^2/(\hat{\lambda}_{d}\lambda_{d})^{3}.
\] 
Since $\|\hat{\ccC}-\ccC\|_{\cS}=o_P(1)$ we know that $\hat{\lambda}_d=\lambda_d+ o_P(1)$ and $\hat{\lambda}_{d+1}=\lambda_{d+1}+ o_P(1)$ and the proof is complete since $\lambda_d>\lambda_{d+1}\geq 0$.
\end{proof} 

\begin{proof}[Proof of \autoref{thm:convH0}]
It holds under the null hypothesis that
\begin{align*}
&|\max_{1\leq k <n}|S_k(\bbeta)|_{\Sigma} - \max_{1\leq k <n} |S_k(\hat{\bbeta})|_{\hat{\Sigma}}|\\
&= |\max_{1\leq k <n}|S_k(\bvarepsilon)|_{\Sigma} - \max_{1\leq k <n} |S_k(\hat{\bvarepsilon})|_{\hat{\Sigma}}|\\
&=|\max_{1\leq k <n}\|\ccC^{(d)} S_k (\varepsilon) \| -  \max_{1\leq k <n} \|\hat{\ccC}^{(d)} S_k(\varepsilon)\|| \\
&\leq \max_{1\leq k <n} \|[\ccC^{(d)} - \hat{\ccC}^{(d)}]S_k(\varepsilon)\|\\
&\leq \|\ccC^{(d)} - \hat{\ccC}^{(d)}\|_\cL \max_{1\leq k <n}\|S_k(\varepsilon)\|\leq \|\ccC^{(d)} - \hat{\ccC}^{(d)}\|_\cS \max_{1\leq k <n}\|S_k(\varepsilon)\|,
\end{align*}
for some $c>0$, where $\|\cdot\|_{\cL}$ is the operator norm. The assertion follows on using \autoref{prop:bounds}.
\end{proof}

To investigate the behavior under the alternative we need a weighted law of large numbers. The proof is obvious and thus omitted.
\begin{proposition}\label{diss:prop:weaklpm}		
Let $\{\varepsilon_i\}_{i\in\mZ}$ be an \,\(H\)-valued, centered time series and let $\{a_{n,i}\}_{n,i\in\mN}$ be an array of scalars such that $\max_{1\leq i\leq n}|a_{n,i}|=\cO(n^{-1})$.
Given $\sum_{i,j=1}^n |E\langle \varepsilon_i,\varepsilon_j\rangle|=\cO(n)$ it holds that $E\|\sum_{i=1}^n a_{n,i}\varepsilon_i\|^2=\cO(n^{-1})$. 
\end{proposition}	 
The assumption of \autoref{diss:prop:weaklpm} on the noise is fulfilled under $\ccL^2$-$m$-approx\-imability according, e.g., to \citet{lukasz2015}. 
\\
\\
Subsequently, we make use of the following functions
\begin{align} 
	\ccG(g) &= \int_0^1 g^2(y)dy - \Big(\int_0^1 g(y) dy \Big)^2,\\
	\intertext{and}
	\cG(x) &= \int_0^x g(y)dy - x \int_0^1 g(y)dy.  \label{eq:integralsup} 
\end{align}
Furthermore, for convenience, we set 
\begin{equation}\label{eq:defbetas}
	\beta_{n,i}=g(i/n)-\int_0^1g(x)dx
\end{equation}
and recall that under $H_A$ the function $g$ is non-constant in which case $\ccG(g)>0$ and 
\begin{equation}\label{eq:cGlarge}
	\ccS:=\sup_{0\leq x\leq 1}|\cG(x)|>0,
\end{equation}
as well. For the sake of a more compact presentation we restrict ourselves to Lipschitz continuous trend alternatives. 
The piecewise case may be treated similarly.
\begin{lem}\label{diss:lem:coefficient_convergence}		
Assume that $g$ is Lipschitz continuous and set the $\beta_{n,i}$ according to \eqref{eq:defbetas}. It holds that 
\[
	\lim_{n\rightarrow\infty}\max_{0\leq r\leq h}\Big|\frac{1}{n}\sum_{i=1}^{n-r}\beta_{n,i}\beta_{n,i+r}-\ccG(g)\Big|=0, 
\]
where  $h\rightarrow\infty$ with $h=o(n)$, as $n\rightarrow\infty.$\!
\end{lem}

		\begin{proof}[Proof of  \autoref{diss:lem:coefficient_convergence}]
		Using the Lipschitz continuity of $g$ we get that
		\begin{align}
			&\max_{0\leq r \leq h}|\sum_{i=1}^{n-r}\beta_{n,i}\beta_{n,i+r}/n-\ccG(g)|\nonumber\\
			&\leq  \max_{0\leq r \leq h} |\sum_{i=1}^{n-r}\beta_{n,i}\beta_{n,i}/n-\ccG(g)\Big|  + \max_{0\leq r \leq h} \bigg(\sum_{i=1}^{n-r}|\beta_{n,i}||\beta_{n,i}-\beta_{n,i+r}|/n\bigg) \label{diss:eq:abrupt_cov_case}\\
			&\leq   |\sum_{i=1}^{n}\beta_{n,i}\beta_{n,i}/n-\ccG(g)\Big| + \max_{0\leq r \leq h} |\sum_{i=n-r+1}^{n}\beta_{n,i}\beta_{n,i}\Big|/n\nonumber \\
			&\quad +  \bigg( \max_{0\leq r \leq h} \max_{1\leq i\leq n-r}|\beta_{n,i}-\beta_{n,i+r}| \bigg)  \bigg(\sum_{i=1}^{n}|\beta_{n,i}|/n\bigg)=\cO(h/n)\nonumber
		\end{align} 
		 which completes the proof.
		\end{proof}

 	The next theorem is an extension of Lemma B.2 of \citet{horvath2013testing}.
	 \begin{theorem}\label{diss:theorem:intermediate_conv_alt_bothconcepts}
	 Let all assumptions of \autoref{thm:alternative} hold true and assume that $\|\ccC-\hat{\ccC}\|_{\cS}=o_P(1)$ under $H_0$. It holds under $H_A$ that, 
	 \begin{equation}\label{diss:eq:decomp_as}
	 	\Big\|\hat{\ccC} -  s_n \ccD\Big\|_{\cS}=o_P(h),
	 \end{equation}
	 as $n\rightarrow\infty$, with $s_n = \ccG(g)\sum_{r=-n}^{n}\ccK(r/h)$  and $\ccD=\Delta\otimes \Delta$.
	 \end{theorem}
	  
	 \begin{proof}[Proof of \autoref{diss:theorem:intermediate_conv_alt_bothconcepts}] We sketch the proof only for a Lipschitz continuous $g$.
	Set the $\beta_{n,i}$ according to \eqref{eq:defbetas} and observe that
	\begin{align}\label{diss:eq:conv_uniform_ref}
	\begin{split}
		\eta_{i+r}-\bar{\eta}_n &= \varepsilon_{i+r} -\bar{\varepsilon}_n + \beta_{n,i} \Delta + r_n, 
	\end{split}
	\end{align}
	where $r_n=o(1)$ is uniform in $i$ and $r$. Let us decompose as follows
		\begin{align} 
	{(\eta_i-\bar{\eta}_n)}\otimes{(\eta_{i+r}-\bar{\eta}_n)}= &  {(\varepsilon_i-\bar{\varepsilon}_n)}\otimes{(\varepsilon_{i+r}-\bar{\varepsilon}_n)} +  \beta_{n,i}\beta_{n,i+r}(\Delta \otimes \Delta)\\ 
	&\qquad+ \beta_{n,i+r} (\varepsilon_{i} \otimes \Delta)  +   \beta_{n,i} (\Delta \otimes \varepsilon_{i+r}) +  R_n
 	\end{align}
 	with a remainder $R_n$ which is of order $o_P(h)$ due to $\sum_{i=1}^n \beta_{n,i}\varepsilon_i/n=o_P(1)$ and the latter being a consequence of \autoref{diss:prop:weaklpm}.
 	The proof is finished on showing that the term $\beta_{n,i}\beta_{n,i+r}(\Delta \otimes \Delta)$ is dominating.	It holds that
	\begin{align}
	&\Big\|\sum_{r=1}^{n}\ccK(r/h)  \sum_{i=1}^{n-r} \beta_{n,i+r} \Big[ \varepsilon_{i} \otimes \Delta\Big]/n\Big\|_{\cS}\nonumber\\
	&\leq   \lfloor ch\rfloor \max_{1\leq r\leq \lfloor ch\rfloor} \Big\| \sum_{i=1}^{n-r} \beta_{n,i} \varepsilon_{i}\Big\|/n  + \sum_{r=1}^{ \lfloor ch\rfloor} \bigg(\sum_{i=1}^{n-r}|\beta_{n,i}-\beta_{n,i+r}|\| \varepsilon_{i}\|/n\bigg) \label{diss:eq:more_delicate1}\\
	&\leq  \lfloor ch\rfloor \max_{1\leq r\leq  \lfloor ch\rfloor} \Big\| \sum_{i=1}^{n-r} \beta_{n,i} \varepsilon_{i}\Big\|/n \nonumber \\
	&\quad+  \lfloor ch\rfloor \max_{1\leq r \leq  \lfloor ch\rfloor}\Big(\max_{1\leq i\leq n-r}|\beta_{n,i}-\beta_{n,i+r}|\Big)  \bigg(\sum_{i=1}^{n}\| \varepsilon_{i}\|/n\bigg)=o_P(h)\nonumber
	\end{align} 
	for some $c>0$. The rate $o_P(h)$ is shown as follows: First, using \autoref{diss:prop:weaklpm} and $\sum_{i=n-\lfloor c h \rfloor +1}^n |\beta_{n,i}|E\| \varepsilon_{i}\|=\cO(h)$ we observe that
	\begin{align*}
	 &\max_{1\leq r\leq \lfloor ch\rfloor} \Big\| \sum_{i=1}^{n-r} \beta_{n,i} \varepsilon_{i}\Big\|/n \leq \| \sum_{i=1}^{n} \beta_{n,i} \varepsilon_{i}\Big\|/n + \!\sum_{{i=n- \lfloor ch\rfloor+1}}^n |\beta_{n,i}| \| \varepsilon_{i}\|/n =\cO_P(h/n).
	\end{align*} 
	Second, by the Lipschitz continuity and the law of large numbers we get
	\begin{align}
	\max_{1\leq r \leq  \lfloor ch\rfloor}\Big(\max_{1\leq i\leq n-r}|\beta_{n,i}-\beta_{n,i+r}|\Big)  \bigg(\sum_{i=1}^{n}\| \varepsilon_{i}\|/n\bigg)=\cO(h/n).
	\end{align} 
	 An application of \autoref{diss:lem:coefficient_convergence} completes the proof.
	\end{proof}

		For the next proofs we note that we have uniform convergence towards partial sums of expectations, in view of \eqref{eq:funcsums}, and they converge uniformly via piecewise Lipschitz continuity of $g$ towards an integral expression involving \eqref{eq:integralsup}: 
 \begin{equation}\label{eq:mainalt} 
 	 \sup_{0\leq x \leq 1}\|S_{\frac{\lfloor n x \rfloor}{n}}(\eta)/n^{1/2} - \cG(x)\Delta \|=\cO_P(n^{-1/2})=o_P(1). 
 \end{equation}
 
	\begin{proof}[Proof of \autoref{thm:alternativeGeneral}]
	We consider the projection in the $k$-th direction
\[
	\|\hat{\ccC}^{(d)} S_k(\eta)\|\geq |\langle S_k(\eta)/n^{1/2},\hat{v}_k\rangle| |n/\hat{\lambda}_k|^{1/2}.
\]
A combination of  \eqref{eq:cGlarge}, \eqref{eq:mainalt} and of the assumptions on $\hat{v}_k$ and $\hat{\lambda}_k$ yields the assertion. 
\end{proof}
	
\begin{proof}[Proof of \autoref{thm:alternative}]
Similar to the proof of \autoref{thm:alternativeGeneral} we may consider the relation 
\[
	(h/n)^{1/2}\|\hat{\ccC}^{(d)}_{B} S_k(\eta)\|= \|h^{1/2}\hat{\ccC}^{(d)}_{B} S_k(\eta)/n^{1/2}\|\geq \|h^{1/2}\hat{\ccC}^{(1)}_{B} S_k(\eta)/n^{1/2}\|.
\]
From \autoref{diss:theorem:intermediate_conv_alt_bothconcepts} and the Proposition 4 of \citet{lukasz2015} we obtain that $\|\hat{\ccC}_{B}/h-\kappa \ccD\|_{\cS}=o_P(1)$ with $\kappa=\ccG(g)\int_{-\infty}^\infty \ccK(x)dx$ and thus $\hat{\lambda}_1/h=\kappa+o_P(1)$, $\hat{\lambda}_2=o_P(1)$ and also $\|\hat{v}_1-\hat{s}\Delta\|=o_P(1)$ with $\hat{s}=\sign \langle \hat{v}_1,\Delta \rangle$.
Hence, it is $\|h^{1/2}\hat{\ccC}^{(1)}_{B}-\kappa^{-1/2}\ccD\|_{\cS}=o_P(1)$. 
Moreover, we know from \eqref{eq:mainalt} that 
\[
	\max_{1\leq k \leq n}\|\ccD S_{k}(\eta)/n^{1/2}\|=\ccS+o_P(1)
\] 
with $\ccS>0$ and the proof is complete.  
\end{proof}

\begin{proof}[Proof of \autoref{cor:nullAligned}]
The assertion follows immediately from \autoref{thm:convH0} observing that 
\begin{align*}
	|\hat{\ccT}^{(d)}- \hat{\ccT}^{(d)\prime}|&\leq \max_{1\leq k <n}|S_k(\hat{\bvarepsilon})-S_k(\hat{\bvarepsilon}^\prime)|_{\hat{{\Sigma}}}\\
	 &= \max_{1\leq k <n}\|S_k(\varepsilon)\|\|\hat{v}_1-\hat{v}_1^\prime\|/|\hat{\lambda}_1|^{1/2}\\
	  &\leq \max_{1\leq k <n}\|S_k(\varepsilon)\|\frac{\|\hat{v}_1\||1-\|\hat{v}_1+n^\gamma \hat{s}\hat{u}\||+\|n^\gamma\hat{u}\| }{\|\hat{v}_1+n^\gamma \hat{s}\hat{u}\|}\frac{1}{|\hat{\lambda}_1|^{1/2}}\\
	 &=o_P(1),
\end{align*}
since $\|\hat{v}_1\|=1$ and $n^\gamma \hat{u}=\cO_P(n^{-1/2+\gamma})=o_P(1)$ in view of \eqref{eq:funcsums}. We used that $n^\gamma \hat{u}=\cO_P(n^{-1/2+\gamma})=o_P(1)$ since $\gamma\in(0,1/2)$ and that $\hat{\lambda}_1=\lambda_1+o_P(1)$ if $\|\ccC-\hat{\ccC}\|_{\cS}=o_P(1)$.
\end{proof}

\begin{proof}[Proof of \autoref{thm:alternativeAligned}] 
Relation \eqref{eq:mainalt} yields that 
 \[
 	\|\hat{u}\|=\max_{1\leq k <n}\|S_k(\eta)/n^{1/2}\|=\ccS+o_P(1)
\] 
and also that $\|\hat{v}_1/n^{\gamma} + \hat{s}\hat{u}\|=\ccS +o_P(1)$ with $\ccS>0$.
   Due to $\smash{\hat{\ccT}_n^{(d)\prime}\geq \hat{\ccT}_n^{(1)\prime}}$, it is sufficient to consider the behavior of the statistic in the first direction which corresponds to
\begin{align}\label{eq:step1} 
\begin{split}
\hat{\ccT}_n^{(1)\prime} &=\max_{1\leq k < n}|\langle S_{k}(\eta)/n^{1/2},\hat{v}_1/n^{\gamma} + \hat{s}\hat{u}\rangle|\Big|\frac{|n/\hat{\lambda}_1|^{1/2}}{\|\hat{v}_1/n^{\gamma} + \hat{s}\hat{u}\|}\\
&\geq  \Big|\max\limits_{1\leq k < n}|\langle S_{k}(\eta)/n^{1/2},\hat{v}_1/n^{\gamma}\rangle| \\
&\quad -\max\limits_{1\leq k<n}|\langle S_{k}(\eta)/n^{1/2}, \hat{u}\rangle|\Big|\frac{|n/\hat{\lambda}_1|^{1/2}}{\|\hat{v}_1/n^{\gamma} + \hat{s}\hat{u}\|}.\\
\end{split}
\end{align}
Via \eqref{eq:mainalt} we observe, for one thing, that
\begin{align} 
	  \max_{1\leq k < n}\frac{|\langle S_{k}(\eta)/n^{1/2},\hat{v}_1/n^{\gamma}\rangle|}{\|\hat{v}_1/ n^{\gamma}+\hat{s} \hat{u}\|}
	  &\leq \frac{\|\hat{u}\|}{\|\hat{v}_1/n^{\gamma} +\hat{s}\hat{u}\|} \frac{\|\hat{v}_1\|}{n^{\gamma}}=\cO_P(n^{-\gamma})
\end{align}
and, for another thing, by evaluating at $k=\hat{k}$, that
\begin{equation}\label{eq:step3}
	\max_{1\leq k < n}\frac{|\langle S_{k}(\eta)/n^{1/2},\hat{u}\rangle|}{\|\hat{v}_1/n^{\gamma} + \hat{s}\hat{u}\|}\geq \|\hat{u}\|\frac{\|\hat{u}\|}{\|\hat{v}_1/n^\gamma + \hat{u}\|}=\ccS+o_P(1),
\end{equation}
The claim follows by \eqref{eq:cGlarge} and by recalling the assumption that $P(|n/\hat{\lambda}_1|>1/c)=P(|\hat{\lambda}_1/n|<c)\rightarrow 1$, for any  $c>0$.

In contrast to the proof of \autoref{thm:alternative} we did not \textit{explicitly} use the asymptotic behavior of $\hat{v}_1^\prime$.
\end{proof}

 
	\vspace{1cm}
	{\footnotesize
\begingroup 
\setstretch{0.3}
\microtypesetup{tracking=false}
 
\par
\endgroup 
}

\end{document}